\documentclass[12pt,a4paper]{article}

\makeatletter
\def\thickhrulefill{\leavevmode \leaders \hrule height 1ex \hfill \kern \z@}
\def\maketitle{
  \vspace*{0\p@}%
  {\parindent \z@ \centering \reset@font%
        \thickhrulefill \quad \scshape \@author \ -\ \@date \quad \thickhrulefill%
        \par\nobreak%
        \vspace*{10\p@}%
        \interlinepenalty\@M
        \hrule
        \vspace*{10\p@}%
        \Large \bfseries \@title \par\nobreak
        \par
        \vspace*{10\p@}%
        \hrule
    \vskip 30\p@
}} \makeatother

\makeatletter
\def\blfootnote{\xdef\@thefnmark{}\@footnotetext}
\makeatother

\usepackage{pifont,fancybox,enumerate,multirow,bigstrut,amsmath,amssymb,amsfonts,graphicx,xypic}
\usepackage{amsthm,url,hyperref,xcolor,titlesec}
\usepackage[french]{babel}
\usepackage[all]{xy}

\numberwithin{equation}{section}

\newtheoremstyle{rema}%
{10pt}%
{10pt}%
{}%
{}%
{\itshape}%
{\ --}%
{0.5em}%
{}

\newtheorem{lemma}{Lemme}[section]
\newtheorem{prop}[lemma]{Proposition}
\newtheorem{thm}[lemma]{Th\'eor\`eme}
\newtheorem{corr}[lemma]{Corollaire}
\newtheorem{defin}[lemma]{D\'efinition}
\newtheorem{conj}[lemma]{Conjecture}
\newtheorem{ques}[lemma]{Question}

\theoremstyle{rema}
\newtheorem{rem}[lemma]{Remarque}

\newcommand{\Div}{\ensuremath{\text{Div}}}

\newcommand{\divi}{\ensuremath{\text{div}}}
\newcommand{\Pic}{\ensuremath{\text{Pic}}}

\newcommand{\Gal}{\ensuremath{\text{Gal}}}
\newcommand{\Hom}{\ensuremath{\text{Hom}}}

\newcommand{\Inv}{\ensuremath{\text{Inv}}}
\newcommand{\id}{\ensuremath{\text{id}}}
\newcommand{\im}{\ensuremath{\text{im}}}

\newcommand{\SK}{\ensuremath{\text{SK}}}
\newcommand{\SKb}{\ensuremath{\textbf{SK}}}
\newcommand{\car}{\ensuremath{\text{car}}}
\newcommand{\ind}{\ensuremath{\text{ind}}}
\newcommand{\Frac}{\ensuremath{\text{Frac}}}
\newcommand{\Br}{\ensuremath{\text{Br}}}
\newcommand{\SL}{\ensuremath{\text{SL}}}
\newcommand{\SLb}{\ensuremath{\textbf{SL}}}
\newcommand{\spec}{\ensuremath{\text{Spec}}}

\newcommand{\Lbar}{\ensuremath{\overline{L}}}
\newcommand{\Fbar}{\ensuremath{\overline{F}}}
\newcommand{\Dbar}{\ensuremath{\overline{D}}}

\newcommand{\GB}{\ensuremath{\textbf{G}}}
\newcommand{\GBcal}{\ensuremath{{\pmb{\mathcal{G}}}}}
\newcommand{\Gbar}{\ensuremath{\overline{G}}}
\newcommand{\GBbar}{\ensuremath{\overline{\textbf{G}}}}
\newcommand{\eenB}{\ensuremath{\textbf{1}}}
\newcommand{\dlog}{\ensuremath{\text{dlog}}}
\newcommand{\Nrd}{\ensuremath{\text{Nrd}}}
\newcommand{\CH}{\ensuremath{\text{CH}}}
\newcommand{\kcorps}{\ensuremath{k\text{-}\mathfrak{corps}}}
\newcommand{\rcorps}{\ensuremath{R\text{-}\mathfrak{corps}}}
\newcommand{\Kcorps}{\ensuremath{K\text{-}\mathfrak{corps}}}

\newcommand{\Tr}{\ensuremath{\text{Tr}}}
\newcommand{\verylongrightarrow}{\relbar\joinrel\longrightarrow}

\newcommand{\cd}{\ensuremath{\text{cd}}}

\newcommand{\Acal}{\ensuremath{\mathcal{A }}}
\newcommand{\Bcal}{\ensuremath{\mathcal{B }}}

\newcommand{\Hcal}{\ensuremath{\mathcal{H }}}

\newcommand{\Ocal}{\ensuremath{\mathcal{O }}}
\newcommand{\Pcal}{\ensuremath{\mathcal{P }}}

\newcommand{\Xcal}{\ensuremath{\mathcal{X }}}

\newcommand{\Hb}{\ensuremath{\mathbb{H }}}

\newcommand{\Zb}{\ensuremath{\mathbb{Z }}}

\titleformat{\subsubsection}[runin]
  {\itshape}
  {\thesubsubsection}
  {5pt }
  {}[\ --]

\renewcommand{\thesubsubsection}{(\alph{subsubsection})}

\title{L'invariant de Suslin en caract\'eristique positive}
\author{Tim Wouters}

\begin{document}

\maketitle

\blfootnote{\ \\
\textit{Adresse:} Tim Wouters, K.U.Leuven, Departement Wiskunde, Celestijnenlaan 200B bus 2400, B-3001 Leuven, Belgique --
tim@wouters.in
\\ 
\textit{Classification (AMS) par sujet 2010:} 19B99 (12G05, 16K50, 17C20)\\
\textit{Mots cl\'es:} Groupe de Whitehead r\'eduit -- Modules de cycles -- Invariants cohomologiques}

\begin{abstract} 
\addcontentsline{toc}{section}{R\'esum\'e} 
Pour une $k$-alg\`ebre simple centrale $A$ d'indice inversible dans $k$, Suslin a d\'efini un invariant cohomologique de $\SKb_1(A)$ \cite{suslin}.
Dans ce texte, nous g\'en\'eralisons cet invariant  \`a toute $k$-alg\`ebre simple centrale par
un rel\`evement de la caract\'eristique positive \`a la caract\'eristique
0.  Pour pouvoir d\'efinir cet invariant, on a besoin des groupes de cohomologie des diff\'erentielles 
logarithmiques de Kato \cite{katogalcoh}.
\end{abstract}

\renewcommand{\abstractname}{Abstract}

\begin{abstract}
For  a central simple $k$-algebra $A$ with $\ind_k(A)\in k^\times$,  Suslin defined a 
cohomological invariant for $\SKb_1(A)$ \cite{suslin}.  In this text, we generalise his invariant to
any central simple $k$-algebra using a lift from positive characteristic to characteristic 0.
To be able to define the invariant, we  use Kato's cohomology of logarithmic differentials \cite{katogalcoh}.
\end{abstract}

\section{Introduction}

Soient $k$ un corps, $A$ une $k$-alg\`ebre simple centrale et $\SLb_1(A)$
le groupe alg\'{e}brique lin\'eaire usuel que nous consid\'erons comme foncteur, donc d\'efini pour chaque extension de corps $F$ de $k$
par 
\[ \SLb_1(A)(F)=\SL_1(A\otimes_k F)=\{ a \in A\otimes_k F \, | \, \Nrd_{A\otimes_k F/F} (a)=1 \}. \] 
Soit de plus $\SKb_1(A)$ le foncteur en groupes qui associe \`a chaque extension de corps $F$ de $k$
le \textit{groupe de Whitehead r\'eduit} de $A_F:=A\otimes_k F$; i.e. 
\[ \SKb_1(A)(F)=\SK_1(A_F)\cong \SL_1(A_F)/[A_F^\times,A_F^\times]. \] 
Puisque $\SKb_1(A)$ est donc li\'e au groupe alg\'ebrique lin\'eaire $\SLb_1(A)$,
il vaut la peine de l'etudier, certainement parce que Platonov nous a 
rassur\'e qu'il n'est pas forc\'ement trivial \cite[Thm. 5.19]{sk1niettriv}.  
La question de la trivialit\'e a \'et\'e ind\'ependamment  pos\'ee par Tannaka et Artin en 1943. %
Pendant plus de 30 ann\'ees, on a essay\'e de la prouver -- \textit{le probl\`eme de
Tannaka-Artin}  \cite{nakmat,wang}.  $\SKb_1(A)$ est quand-m\^eme trivial si $\ind_k(A)$ est
sans facteurs carr\'es (Th\'eor\`eme de Wang \cite{wang}).  Suslin a conjectur\'e la r\'eciproque \cite{suslinconj}. R\'ecemment, 
Merkurjev a d\'emontr\'e que cette conjecture vaut si $4\mid \ind_k(A)$ \cite{mersuslinbiquat},
et Rehman-Tikhonov-Yanchevski\u{\i} ont d\'emontr\'e qu'il suffit de d\'emontrer la conjecture pour des
alg\`ebres \`a division cycliques \cite[Thm. 0.19]{rehmanea}.

Une fa\c con d'\'etudier $\SKb_1(A)$ est de construire des \textit{invariants cohomologiques}.
Inspir\'e par Platonov \cite{sk1niettriv}, Suslin a construit un invariant
cohomologique de $\SKb_1(A)$ dans le cas o\`u $n=\ind_k(A)\in k^\times$ \cite{suslin}.
Si $[A]\in H^2(k,\mu_n)$ est la classe de Brauer de $A$, l'invariant
de Suslin est un morphisme fonctoriel en le corps (la fonctorialit\'e est en effet inh\'erente \`a la d\'efinition d'un invariant):
\[ \rho_{\text{Sus},A} : \SKb_1 (A)(k) \to H^4(k, \mu_n^{\otimes 3} ) / (H^2(k,\mu_n^{\otimes 2}) \cup [A]).\]
Pour un aper\c{c}u plus d\'etaill\'e de cette histoire, voir \cite[\S 2]{gillebourbakifr}.
  La \textit{conjecture de Bloch-Kato}, r\'ecemment d\'emontr\'ee par Voevodsky-Rost-Weibel
\cite{blochkato,voevodblk,rostblk,weibelblk}, donne que \textit{le symbole galoisien} $h_{n,k}^i:K_i^M(k)/nK_i^M(k) \to H^i(k,\mu_n^{\otimes i})$ d\'ecrit les $K$-groupes de Milnor en termes de la cohomologie galoisienne ($i\geq 0$).  On sait donc r\'e\'ecrire l'image de $\rho_{Sus}$ avec une structure de $K_n^M(k)$-module qui est 
induite par le symbole galoisien.

Le but de ce texte est de g\'en\'eraliser l'invariant de Suslin pour toute $k$-alg\`ebre simple centrale, utilisant
qu'il existe pour toute alg\`ebre simple centrale en caract\'eristique 0.  Si $k$ est un corps de caract\'eristique 
$p>0$, on le consid\`ere comme corps r\'esiduel d'un anneau  $R$ complet de valuation discr\`ete $v$ 
avec corps des fractions $K$ de caract\'eristique 0.  
On sait relever $A$ en une $R$-alg\`ebre d'Azumaya $B$ tel que $B\otimes_R k\cong A$. Alors $B_K:=B\otimes_R K$ devient
une $K$-alg\`ebre simple centrale, et Suslin a donc construit un
invariant cohomologique de $\SKb_1(B_K)$.  Vu qu'un r\'esultat de Platonov \cite[Cor. 3.13]{sk1niettriv} induit un isomorphisme  $\SKb_1(B_K)(K)\cong \SKb_1(A)(k)$, on peut essayer d'en d\'eduire des invariants cohomologiques
de $\SKb_1(A)$.  La subtilit\'e ici est que l'isomorphisme de Platonov se d\'eroule au niveau des groupes,
et pas au niveau des foncteurs, ce qui cause quand m\^eme quelques soucis.  De plus, la d\'efinition
d\'epend du choix de $R$.  Heureusement, on sait passer outre
en utilisant des \textit{anneaux de Cohen}.

On commence par le cas o\`u $\car (k)>0$ et l'indice de $A$ est inversible dans $k$ (\textit{le cas mod\'er\'e}).  L'invariant de Suslin
est en effet d\'ej\`a d\'efini dans ce cas, mais le rel\`evement permet aussi de construire un invariant
en caract\'eristique positive \`a partir de tout invariant de $\SKb_1(B_K)$ (en caract\'eristique 0) du m\^eme type. De plus, cela sera
un bon \'echauffement pour le cas sauvage.
Avant d'\'etablir le rel\`evement, on d\'ecrit les groupes de cohomologie utilis\'es
(Section \ref{sec:cohom}).  
On les ins\`ere dans le formalisme des modules de cycles de Rost \cite{Rostmodcyc} afin de 
pouvoir utiliser une caract\'erisation tr\`es utile des invariants de Merkurjev \cite[Lem. 2.1, Thm. 2.3]{invalggroup}.  
Puis, on a tous les ingr\'edients en vue de relever les invariants (Section \ref{sec:relevmod}).

Apr\`es ce rel\`evement, on continue au cas o\`u l'indice de $A$ peut \^etre non inversible dans $k$ (\textit{le cas sauvage}).  Afin de construire un invariant, il faut utiliser les groupes de cohomologie de diff\'erentielles logarithmiques
de Kato \cite{katogalcoh} qui g\'en\'eralisent la cohomologie galoisienne.  Eux aussi sont
servis d'une structure de $K_n^M(k)$-modules comme les groupes de cohomologie galoisienne ordinaires.  
Un r\'esultat de Kahn \cite{kahnweight2mot} nous fournit
 le dernier ingr\'edient pour pouvoir effectuer un rel\`evement et de telle fa\c{c}on obtenir l'invariant d\'esir\'e (Section \ref{sec:sauvage}).  Utilisant la validit\'e de la conjecture de Suslin pour les biquaternions, nous pouvons nous
 rassurer que cet invariant relev\'e est non trivial.  Nous finissons ce papier avec quelques remarques: entre autre quelques
consid\'erations sur la
 conjecture de Suslin au regard des r\'esultats obtenus.

\subsubsection*{Notations}
\addcontentsline{toc}{subsubsection}{Notations} 
  Fixons quelques notations durant tout ce texte.  
   
\begin{itemize}
\item  Pour un corps $k$, on note $k_s$ une cl\^oture
s\'eparable et $\Gamma_k=\Gal(k_s/k)$ le groupe de Galois absolu.
\item Pour un entier $m>0$ et un corps $k$, on note le $\Gamma_k$-module des racines $m$-i\`emes d'unit\'e de $k_s$ par $\mu_m$.
\item Les groupes de cohomologies utilis\'es --sauf mention expresse-- sont des groupes de cohomologie galoisienne (ou \'etale) .
\item $\phantom{ }_m \Br(k)$ est la $m$-i\`eme partie de torsion du groupe de Brauer de $k$ ($m>0$ un entier).
\item Si $F$ est un corps munie d'une valuation discr\`ete $v$, l'anneau de valuation est not\'e par $\Ocal_v$ et le corps r\'esiduel, par $\kappa(v)$.
L'extension maximale non ramifi\'ee de $F$ est not\'ee par $F_{nr}$.  Si $x\in \Ocal_v$, on note $\bar{x}$ son
r\'esidu dans $\kappa(v)$.  On utilise cette notation-ci aussi pour d'autres objets munis de r\'esidus naturels.
Une valuation discr\`ete est toujours suppos\'ee \^etre non triviale (de rang 1 et de groupe des valeurs $\mathbb{Z}$).
\item Soient $A$ une $k$-alg\`ebre simple centrale et $F$ une extension de corps de $k$, alors $A_F:=A\otimes_k F$ est
la $F$-alg\`ebre simple centrale obtenue par extension de base.
\end{itemize}

\subsubsection*{Remerciements}
\addcontentsline{toc}{subsubsection}{Remerciements} 

L'auteur remercie vivement son directeur de th\`ese, Philippe Gille, pour soutenir cet article par ses id\'ees et commentaires tr\`es utiles.  Il
remercie aussi la K.U.Leuven et l'\'Ecole Normale Sup\'erieure (Paris) pour le support
financier et l'hospitalit\'e qui ont r\'ealis\'e tant de visites.
L'auteur est aussi partiellement soutenu par \og Fonds voor Wetenschappelijk Onderzoek Vlaanderen\fg\ (G.0318.06).

\section{Modules de cycles} \label{sec:cohom}
Dans cette section, on commence par d\'ecrire les groupes de cohomologie o\`u l'invariant de
Suslin a ses valeurs.  Il sont des exemples privil\'egi\'es de modules de cycles.  On donne une br\`eve introduction \`a ce formalisme introduit par Rost \cite[\S 1,2]{Rostmodcyc}, et
puis on explique le lien avec les invariants cohomologiques d'un groupe alg\'ebrique d\'ecouvert
par Merkurjev \cite{invalggroup}.

\subsection{Groupes de cohomologie} \label{sec:cohomdef}

Dans toute cette section, soient $F$ un corps de $\car(F)=p\geq 0$ et $m> 0$ un entier 
inversible dans $F$.

\subsubsection{D\'efinition} 
 Soit  $\mu^{\otimes i}_m$ la $i$-i\`eme produit tensoriel de $\mu_m$ comme $\mathbb{Z}/m\mathbb{Z}$-module ($i\geq 0$),
alors on d\'efinit:
\[ H^{i}_m(F):=H^{i}(F,\mu_{m}^{\otimes i}(-1)) \qquad \text{avec} \qquad \mu_{m}^{\otimes i}(-1)=\Hom_{\Gamma_F}(\mu_m,\mu_{m}^{\otimes i}).\] De plus, on pose $H^{i}_m(F)=0$ pour $i<0$. 
Bien \'evidemment, on a $\mu_{m}^{\otimes i+1}(-1)=\mu_m^{\otimes i}$  pour tout $i\geq 0$, et donc $H^{i+1}_m(F)=H^{i+1}(F,\mu_m^{\otimes i})$.\footnote{On utilise le suscrit $i+1$ au lieu de $i$ juste par raisons de traditions et pour \^etre conforme au cas sauvage o\`u il semble un peu plus naturel d'utiliser ce suscrit.}
La suite exacte de Kummer 
\begin{equation}
 1 \to \mu_m(F) \to F_s^\times \overset{m}{\to} F_s^\times \to 1 \label{eq:kummer}
\end{equation}
implique l'interpr\'etation cohomologique bien connue 
du sous-groupe de $m$-torsion du groupe de Brauer $\phantom{ }_m\Br(F) \cong H^2(F,\mu_{m})$.
\subsubsection{Structure de $K_n(F)$-module} \label{sec:symbgal} Consid\'erons les \textit{$K$-groupes de Milnor}\footnote{On utilise surtout des $K$-groupes de Milnor dans la suite.  Donc, pour ne pas alourdir inutilement la notation, on oublie le suscrit $M$ de la notation usuelle $K_n^M(F)$ des $K$-groupes
de Milnor, et on utilise la notation $K_n^Q$ pour les $K$-groupes de Quillen.} $K_n(F)$ pour un entier $n$.  Rappelons que 
\[ K_n(F)=\underset{n \text{ fois}}{\underbrace{F^\times \otimes_{\Zb} \ldots \otimes_{\Zb} F^\times}}/J,\] 
avec $J$ le sous-groupe engendr\'e par des symboles de la forme $x_1\otimes \ldots \otimes x_n$
pour lequel $x_i+x_j=1$ pour $1\leq i < j \leq n$.  Les symboles primitifs sont not\'es abbr\'eg\'es $\{x_1,\ldots,x_n\}$.   
La suite exacte de Kummer \eqref{eq:kummer}
implique au niveau de cohomologie
$ K_1(F)/mK_1(F) = F^\times/(F^\times)^m \cong H^1(F,\mu_m). $
Utilisant le cup-produit, on obtient \textit{le symbole galoisien} 
\begin{equation}
 h_{m,F}^n: K_n(F)/mK_n(F) \to H^n(F,\mu_m^{\otimes n}), \label{eq:symbgal}
\end{equation}
 qui est un isomorphisme (conjecture de Bloch-Kato - Th\'eor\`eme de Voevodsky-Rost-Weibel \cite{blochkato,voevodblk,rostblk,weibelblk}).
La structure de $K_n(F)$-module de $H^{i+1}_m(F)$ pour un entier $i\geq 0$ est d\'efini par le cup produit avec le symbole
galoisien:
\[ K_n(F) \times H^{i+1}_m(F) \to H^{n+i+1}_m(F): (a,b) \mapsto h_{m,F}^n(\bar{a}) \cup b.  \]
On note le produit scalaire par $a\cdot b:=h_{m,F}^n(\bar{a}) \cup b$ pour $a\in K_n(F)$, $\bar{a}$ sa classe dans $K_n(F)/mK_n(F)$ et $b\in H^{i+1}_m(F)$.

\subsubsection{Fl\`eches de r\'esidu} 
Si $v$ est une valuation discr\`ete sur $F$ de corps r\'esiduel $\kappa(v)$, on a une  fl\`eche de r\'esidu pour tout entier $i\geq 0$,
\[ \partial_v^{i+1}: H^{i+1}_m(F) \to H^{i}_m(\kappa(v)),  \] 
qui est bien expliqu\'ee par Serre \cite[\S 6,7]{cohinv}.
De plus, si $F$ est complet pour la valuation, elle a un scindage, parce que tout \'el\'ement $\alpha$ de $H^{i+1}_m(F)$ s'\'ecrit comme
$\alpha_{i+1} + (\pi)\cup \alpha_i$ avec $\pi\in F$ un \'el\'ement uniformisant, $(\pi)$ sa classe dans $F^\times/(F^\times)^m=H^1(F,\mu_m)$ et les $\alpha_j\in H^{j}_m(\kappa(v))$ se plongeant dans $H^{j}_m(F)$
\cite[7.11]{cohinv}.  On a donc une suite exacte scind\'ee int\'eressante:
\begin{equation}
 0 \to H^{i+1}_m(\kappa(v)) \to H^{i+1}_m(F) \overset{\partial_v^{i+1}}{\to} H^{i}_m(\kappa(v)) \to 0. \label{eq:sesrv4}
\end{equation}

\subsubsection{Version relative} \label{sec:relatifmod}
Soit $A$ une $F$-alg\`ebre simple centrale de $\ind_F(A)=n\in F^\times$.
On consid\`ere sa classe de Brauer $[A]\in \phantom{}_n\Br(F)\cong H^2(F,\mu_n)$.
Utilisant le cup-produit de la cohomologie galoisienne, on d\'efinit (pour $i\geq 1$)
\[ H^{i+1}_{n,A}(F) := 
H^{i+1}_n(F)/ \left(H^{i-1}(F,\mu_{n}^{\otimes i-1})\cup [A]\right). 
\]
Si $F$ est munie d'une valuation  discr\`ete $v$, 
on peut \'etendre les r\'esidus de $H_n^{i+1}(F)$ \`a des r\'esidus relatifs.
On note $\bar{A}:= A \otimes_{\Ocal_v} \kappa(v)$, la \textit{$\kappa(v)$-alg\`ebre simple centrale r\'esiduelle} de $A$.  
La description en terme de cocycles explicites 
\cite[\S 6]{cohinv} garantit bien que 
\[ \partial_v^{i+1}(H^{i-1}(F,\mu_{n}^{\otimes i-1})\cup [A]) \subset H^{i-2}(\kappa (v),\mu_{n}^{\otimes i-2})\cup [\bar{A}]. \]
\c{C}a implique donc qu'on a un diagramme commutatif (pour $i\geq 2$):
\begin{equation} \label{eq:diagmod}
\xymatrix{
0 \ar[r] 
& H^{i-1}(\kappa(v), \mu_n^{\otimes i-1}) \ar[r] \ar[d]^{\cup [\overline{A}]}
& H^{i-1}(F, \mu_n^{\otimes i-1}) \ar[r] \ar[d]^{\cup [A]}
& H^{i-2}(\kappa(v), \mu_n^{\otimes i-2}) \ar[r] \ar[d]^{\cup [\overline{A}]}
& 0 
\\
0 \ar[r] 
& H^{i+1}(\kappa(v), \mu_n^{\otimes i}) \ar[r] 
& H^{i+1}(F, \mu_n^{\otimes i}) \ar[r]
& H^{i}(\kappa(v), \mu_n^{\otimes i-1}) \ar[r] 
& 0.
}\end{equation}
 Le lemme du serpent permet donc de construire une suite exacte: 
\begin{equation}
  0 \to H^{i+1}_{n,\bar{A}}(\kappa(v)) \to H^{i+1}_{n,A}(F) \overset{\partial_{v,A}^{i+1}}{\to} H^{i}_{n,\bar{A}}(\kappa(v)) \to 0. \label{eq:suitemod}
\end{equation}
Puisque la suite exacte \eqref{eq:diagmod} est scind\'ee, la suite exacte \eqref{eq:suitemod}
l'est aussi.  \`A noter que d'apr\`es la conjecture de Bloch-Kato et la structure de $K_n(F)$-module, on a une d\'efinition
\'equivalente:
\begin{equation} \label{eq:defaltmod}
 H^{i+1}_{n,A}(F) = 
H^{i+1}(F,\mu_{n}^{\otimes i})/ \left(K_{i-1}(F)\cdot [A] \right).
\end{equation}
\subsection{D\'efinition} \label{sec:def}

Les propri\'etes communes des groupes de cohomologies $H_n^\ast(F)$ et des groupes de $K$-th\'eorie de Milnor
ont inspir\'es Rost \`a d\'efinir 
une structure formelle qui respecte les m\^emes propri\'et\'es
homologiques \cite[\S 1.2]{Rostmodcyc}.  Nous rappelons en bref ce formalisme des modules de cycles.

\subsubsection{D\'efinition d'un module de cycles} \label{sec:defgenmodcyc}
Soient $R$ un anneau de valuation discr\`ete, $\rcorps$ la cat\'egorie des $R$-corps (des $R$-alg\`ebres qui sont des corps) et $\mathfrak{Ab}$ la cat\'egorie
des groupes ab\'eliens.  Un \textit{module de cycles $M$ de base $R$} est un foncteur
\[ \rcorps \to \mathfrak{Ab} \]
munie d'une graduation $M=(M_j)_{j\geq 0}$ et des donn\'ees D1-D4:\footnote{Si on utilise $M_j$ pour $j<0$, dans une notation libre,  il est 0.} ($E,F$ des objets de $\rcorps$ et $\varphi$ un morphisme de $\rcorps$)
\begin{enumerate}[\bf D1:]
 \item Pour tout $\varphi:F\to E$, on a $\varphi_\ast:M(F)\to M(E)$ de degr\'e 0.
 \item Pour tout $\varphi:F\to E$ fini, on a $\varphi^\ast:M(E)\to M(F)$ de degr\'e 0.
 \item Pour tout $F$, le groupe $M(F)$ a une structure de $K_n(F)$-module tel que \break $K_n(F) \cdot M_m(F) \subset M_{n+m}(F)$ ($n,m\geq 0$ des entiers).
 \item Si $F$ est un $R$-corps de valuation discr\`ete $v$, 
il existe un \textit{r\'esidu} $\partial_v:M(F)\to M(\kappa(v))$ de degr\'e $-1$.
\end{enumerate}
Ces donn\'ees doivent respecter des r\`egles de compatibilit\'e (R1a-R3e) et de g\'eom\'etrie (FD et C) -- voir Appendice \ref{sec:append} et \textit{loc. cit.}  \`A noter que, pour obtenir ses buts, Rost lui-m\^eme met plus de restrictions sur la base $R$, mais il commente qu'il est permis de mod\'erer les conditions (ibid., \S 1, p. 328). 

\subsubsection{La base et la coexistence de deux modules de cycles} \label{sec:coexistence}
Le cas classique est celui o\`u la base est 
un corps.  Dans ce texte, on utilise la terminologie des modules de cycles
pour une base qui est un anneau $R$ de valuation discr\`ete complet avec corps des fractions $K$
et corps r\'esiduel $k$.  Un module de cycles $M$ avec une telle base associe donc \`a toute extension de corps
$L$ de $K$ un groupe gradu\'e $M(L)$ et de m\^eme, \`a toute extension de corps
$\Lbar$ de $k$, un groupe gradu\'e $M(\Lbar)$.  

\`A noter que l'on peut bien restreindre un module de cycles
de base $R$ en un module de cycles de base $K$ ou \'egalement de base $k$ en se
restreignant aux extensions de corps de $K$ ou de $k$.  Un module de cycles $M$ avec base $R$ est
donc rien d'autre que la coexistence de deux modules de cycles avec un corps comme base (not\'es $M|_{k}$ et $M|_{K}$)
avec un lien donn\'e par D4.
Dans la suite, on utilise les modules de cycles de base $R$ afin de faciliter la notation et de travailler dans un cadre plus g\'en\'eral.  N\'eanmoins, on pourrait reformuler les arguments,
travailler avec deux modules de cycles et utiliser ces liens comme ext\'erieurs aux modules de cycles m\^emes: comme
des donn\'ees suppl\'ementaires.  

\subsubsection{Complexe de Gersten} \label{sec:gersten}
Soient $F$ un $R$-corps, $X$ une $F$-vari\'et\'e et $M$ un module de cycles, alors l'existence des r\'esidus et
des r\`egles de modules de cycles induisent
un complexe de cycles de Gersten $C_\ast(X,M_j)$ \cite[\S 3.3]{Rostmodcyc} ($i,j\geq 0$):
\[ \ldots \to  \oplus_{x\in X^{(i-1)}} M_{j-i+1} (F(x)) \overset{\partial^{i-1}}{\to} 
 \oplus_{x\in X^{(i)}} M_{j-i} (F(x)) \overset{\partial^i}{\to} 
\oplus_{x\in X^{(i+1)}} M_{j-i-1} (F(x)) \to \ldots, \]
o\`u $X^{(i)}$ d\'esigne l'ensemble des point de codimension $i$ dans $X$ et $F(x)$ est le corps
r\'esiduel de $x$, un point de codimension $i$.  La fl\`eche $\partial^i$ est bien la somme des r\'esidus induits
par les valuations associ\'ees \`a un point de codimension 1 de $X^{(i)}$.
L'homologie de ce complexe au cran $i$ est not\'ee $A^i(X,M_j)$.

\subsubsection{Exemples privil\'egi\'es} \label{sec:exmodcyc}
Lions cette section \`a la pr\'ec\'edente.  Soient $R$ un anneau de valuation discr\`ete complet
de corps des fractions $K$ et corps r\'esiduel $k$ et pour une $K$-alg\`ebre simple centrale $A$ de $\ind_K(A)=n$ tel que $n\in K^\times$ et $n\in k^\times$.  Alors, les foncteurs
\begin{eqnarray*}
\Hcal^\ast_m:=(\Hcal^i_m)_{i\geq 0} : \rcorps \to \mathfrak{Ab} &: &F \mapsto (H^{i}_m(F))_{i\geq 0} \text{ et} \\
\Hcal^\ast_{n,\Acal}:=(\Hcal^i_{n,\Acal})_{i\geq 2} : \rcorps \to \mathfrak{Ab}& :& 
F \mapsto  \left(H^{i}_{n,A_F}(F)\right)_{i\geq 2} 
\end{eqnarray*}
sont bien des modules de cycles o\`u  $A_F:=A\otimes_R F$.
\`A noter que pour une extension $F$ de $K$, la $K$-alg\`ebre $A_F$ est 
de nouveau simple centrale.  On sait que $\ind_k (\overline{A})=n$ avec $\overline{A}=A\otimes_R k$ (voir d\'emonstration
du Corollaire \ref{corr:sk1iso}). Alors pour tout $\rcorps$ $F$, $\ind(A_F) | \, n$  et $[A_F]$ se trouve donc bien dans  $\phantom{ }_n \Br(F)$.
La d\'efinition du deuxi\`eme module de cycles est donc  consistante.

La v\'erification des r\`egles du premier cas pour des $R$-corps d'\'egale caract\'eristique est d\'ecrite par Rost \cite[1.11]{Rostmodcyc}.  Le cas de caract\'eristique mixte suit de fa\c{c}on analogue.
Le premier module de cycle induit la v\'erification du second puisque les donn\'ees et r\`egles se continuent au niveau relatif (voir aussi
\eqref{eq:suitemod}).  

Un autre exemple de modules de cycles est la $K$-th\'eorie de Milnor.
 La donn\'ee D1 est d\'efinie de fa\c{c}on \'evidente.  
 Soit
$E$ une extension finie de corps de $F$, alors la donn\'ee D2 est induite par la norme $N_{E/F}$ appliqu\'ee aux g\'en\'erateurs \cite[Ch. I, \S 5]{basstate}.  
 La donn\'ee D3 est d\'efinie par la structure multiplicative des $K$-groupes:
\[ K_n(F) \times K_m(F) \mapsto K_{n+m}(F): (\{ x_1, \ldots , x_n\},\{ y_1, \ldots , y_m\}) \mapsto (\{ x_1, \ldots , x_n,y_1,\ldots y_m\}).\]
 Soit $F$ un corps de valuation discr\`ete $v$, alors le r\'esidu $K_n(F)\to K_{n-1}(\kappa(v))$ --la donn\'ee D4-- est d\'efini par 
\begin{eqnarray*}
\{ \pi, x_2, \ldots ,x_n \} & \mapsto & \{\bar{x}_2,\ldots, \bar{x}_n \}, \\
\{ x_1,x_2,\ldots,x_n \} & \mapsto & 0,
\end{eqnarray*}
 avec $x_1,\ldots,x_n \in \Ocal_v^\times$ et $\pi$ une uniformisante de $F$ \cite[Lem. 2.1]{milnor}.

\subsection{Lien avec les invariants} \label{sec:lieninv}

Dans toute cette section, soient $k$ un corps et $M=(M_j)_{j\geq 0}$ un module de cycles de base $k$.    Merkurjev a d\'ecouvert un lien int\'eressant f\'econd entre les groupes $A^i(\GB,M_j)$ et les invariants cohomologiques d'un $k$-groupe alg\'ebrique $\GB$ dans $M$ de degr\'e $j$.  Expliquons-le,
et commen\c{c}ons par rappeler la notion d'un invariant d'un groupe alg\'ebrique.

\subsubsection{D\'efinition d'un invariant}  
Des invariants sont des objets fonctoriels, donc il faut des foncteurs.  
On consid\`ere un foncteur $\GB:\kcorps \to \mathfrak{Ab}$ (par exemple un groupe alg\'ebrique) et $M_j$ (pour $j\geq 0$) comme foncteur $\kcorps \to \mathfrak{Ab}$.

Un invariant $\rho$ de $\GB$ dans $M$ de degr\'e $j$ est une
transformation naturelle des foncteurs $\GB \to M_j$.  En particulier, pour toute
extension de corps $F$ de $k$, il consiste en un morphisme fonctoriel
\[ \rho_{F}: \GB(F) \to M_j(F). \]
Tous les invariants de $\GB$ dans $M$ de degr\'e $j$ forment un groupe ab\'elien,
not\'e  $\Inv^j(\GB,M)$.  

\subsubsection{Le lien de Merkurjev} \label{sec:lienmerk}
Soit $\GB$ un groupe alg\'ebrique, alors Merkurjev a construit un morphisme injectif
\begin{equation}
 \theta: \Inv^j(\GB,M) \to A^0(\GB,M_j): \rho \mapsto \rho_{K}(\xi), \label{eq:invarmerk}
\end{equation}  
avec $K=k(\GB)$ et $\xi\in \GB(K)$ le point g\'en\'erique de $\GB$.    
Il d\'emontre que l'image est le sous-groupe $A^0(\GB,M_j)_{\text{mult}}$ contenant les \textit{\'el\'ements  multiplicatifs} de $A^0(\GB,M_j)$  \cite[Lem. 2.1 et Thm. 2.3]{invalggroup}.  Ce sont
les \'el\'ements $x\in A^0(\GB,M_j)$ tels que
\[ p_1^\ast (x)+p_2^\ast(x)=m^\ast(x), \]
o\`u $p_1^\ast,p_2^\ast$ et $m^\ast$ sont les morphismes $A^0(\GB,M_j) \to
A^0(\GB \times \GB,M_j)$ induits par les deux projections $p_1,p_2:\GB \times \GB \to \GB$ et la multiplication $m:\GB\times\GB\to \GB$. 

Il d\'emontre  aussi que $A^0(\GB,M_j)_{\text{mult}}\subset \tilde{A}^0(\GB,M_j)$, 
avec $\tilde{A}^0(\GB,M_j)$
le \textit{sous-groupe r\'eduit} de
$A^0(\GB,M_j)$ (ibid., Lem. 1.9).  Le sous-groupe r\'eduit est le noyau du morphisme $u^\ast:A^0(\GB,M_j)\to A^0(\eenB,M_j)$ qui est induit par le morphisme d'unit\'e $u:\eenB \to \GB$.  Le morphisme $u^\ast$ induit m\^eme un scindage
$A^0(\GB,M_j) \cong \tilde{A}^0(\GB,M_j) \oplus A^0(k,M_j)$.

\subsubsection{L'invariant de Suslin}
Soient $k$ un corps et $A$ une $k$-alg\`ebre simple centrale de norme r\'eduite $\Nrd_{A/k}$,
alors le groupe de Whitehead r\'eduit $\SK_1(A)$ est isomorphe \`a $\SL_1(A)/[A^\times,A^\times]$.
Consid\'erant le foncteur (qui n'est pas un groupe alg\'ebrique)
\[ \SKb_1(A): \kcorps\to \mathfrak{Ab}: F \mapsto \SKb_1(A)(F) = \SK_1(A\otimes_k F),\]
Suslin a introduit
un invariant \cite[\S 3]{suslin}
\begin{equation*} 
\rho_{\text{Sus},A} \in \Inv^4(\SKb_1(A), \Hcal^\ast_{n,\Acal}). 
\end{equation*}
Ici, $\Hcal^\ast_{n,\Acal}=(\Hcal^{j+1}_{n,A})$ est un module de cycles de base $k$.  Faisant de bonnes hypoth\`eses sur $A$,
on peut le voir comme un module de base d'un anneau $R$ complet de valuation discr\`ete restreint \`a son
corps de fractions ou \`a son corps r\'esiduel selon \S \ref{sec:def} \ref{sec:coexistence} (voir aussi \S \ref{sec:relevbase} \ref{sec:cohen}).

Suslin d\'emontre en plus que pour une extension de corps $F$ de $k$, les repr\'esentants de l'image 
de $\SKb_1(A)(F)$ sont des \'elements du noyau de 
$H^4_{n}(F)\to H^4_{n}(F(X))$, o\`u  $X$ d\'esigne la vari\'et\'e de Severi-Brauer associ\'ee \`a $A$ (loc. cit.).

\section{Le rel\`evement, le cas mod\'er\'e} \label{sec:relevmod}

Dans cette section, on utilise un rel\`evement de la caract\'eristique positive \og mod\'er\'ee\fg\ \`a la caract\'eristique 0 
afin de construire des invariants cohomologiques en caract\'eristique positive fond\'es sur l'existence des invariants
en caract\'eristique 0.  En plus, de relier les deux cas, on peut retrouver plus d'information sur la caract\'eristique positive
moder\'ee par la caract\'eristique 0 et vice versa.  

\subsection{La strat\'egie} \label{sec:strategie}

Soient $k$ un corps et $A$ une $k$-alg\`ebre simple centrale.
Pour construire des invariants de $\SKb_1(A)$, on voudrait utiliser l'isomorphisme de Merkurjev \eqref{eq:invarmerk}
qui nous permet d'obtenir des r\'esultats en travaillant avec des complexes de Gersten.   
Malheureusement, cet isomorphisme ne vaut que pour des
groupes alg\'ebriques, et $\SKb_1(A)$ n'en est pas un.  
D\^u \`a la projection $\SLb_1(A)(F)\to  \SLb_1(A)(F)/[A_F^\times ,A_F^\times]\cong \SKb_1(A)(F)$ pour chaque extension de corps $F$ de $k$, on obtient quand-m\^eme une injection de groupes d'invariants.

\begin{lemma} \label{prop:sksl}
Soient $k$ un corps, $A$ une $k$-alg\`ebre simple centrale et $M$ un module de cycles.  La projection de $k$-foncteurs $\pi: \SLb_1(A)\to \SKb_1(A)$ induit pour tout entier $j$ une injection 
\[ \tilde{\pi}: \Inv^j(\SKb_1(A),M) \hookrightarrow \Inv^j(\SLb_1(A),M).\]
\end{lemma}

Nous utilisons donc des invariants de 
$\SLb_1(A)$ qui nous permettent d'utiliser le r\'esultat de Merkurjev.  
En plus, le travail que nous allons effectuer n'est pas
fond\'e sur la d\'efinition de l'invariant de Suslin, mais sur son existence.  Tout autre
invariant analogue peut donc \^etre relev\'e de la m\^eme fa\c{c}on.  

Nous commen\c{c}ons par expliquer la strat\'egie g\'en\'erale du rel\`evement.  Dans cette description, nous ne donnons
pas d'arguments explicites et d\'etaill\'es.  Ceux-ci se trouvent dans les sections suivantes.

\begin{enumerate}[(i)] 
\item \textit{Construire un invariant auxiliaire.}
Soit $k$ un corps de $\car(k)=p>0$ tel que $p$ ne divise pas $\ind_k(A)$.  Il existe 
un anneau $R$ complet de valuation discr\`ete $v$ tel que $k$ soit le corps r\'esiduel et que le corps des fractions $K$ de
$R$ est de caract\'eristique 0.  
Alors, $A$ se rel\`eve en une $R$-alg\`ebre d'Azumaya $B$.  De telle fa\c{c}on, $B_K:=B\otimes_R K$ est
une $K$-alg\`ebre simple centrale.  Afin de construire un invariant dans 
$\Inv^4(\SKb_1(A), \Hcal_{n,A}^\ast)$, nous construisons un invariant auxiliaire $\rho' \in \Inv^3(\SKb_1(A), \Hcal_{n,A}^\ast)$.  
Pour chaque extension de corps $k'$ de $k$, il faut donc d\'efinir un morphisme
\[  \rho'_{k'}:\SKb_1(A)(k') \to H^3_{n,A}(k'). \]
Soit $K'$ un corps complet de valuation discr\`ete $w$ de corps r\'esiduel $k'$ tel que $K'$ est une extension de $K$ et tel que
$w$ prolonge $v$.  Alors, on est donn\'e un isomorphisme  $\SKb_1(B_K)(K')\to \SKb_1(A)(k')$.  \`A partir
d'un invariant $\rho \in \Inv^4(\SKb_1(B_K), \Hcal_{n,B_K}^\ast)$, le r\'esidu du module de cycles \eqref{eq:suitemod} donne un morphisme
\[ \rho'_{k'}:\SKb_1(A)(k') \to H^3_{n,A}(k'). \]
Celui-ci n'est pas n\'ecessairement un invariant, puisque la fonctorialit\'e en les extensions de corps n'est pas
imm\'ediatement obtenue.  En effet, il y a plusieurs mani\`eres afin de
trouver des corps de valuation $K'$ comme ci-dessus.  Pour r\'esoudre ce probl\`eme, on utilise des anneaux de Cohen qui
sont suffisamment canoniques.  
\item \textit{En d\'eduire l'invariant recherch\'e.}
Puisque le r\'esidu des modules de cycles se niche dans une suite exacte fonctorielle \eqref{eq:suitemod}, on obtient
un invariant de degr\'e 4 dans $\Inv^4(\SKb_1(A), \Hcal_{n,A}^\ast) $
d\`es que l'invariant $\rho'$ est trivial.  D'apr\`es le Lemme \ref{prop:sksl},  
pour d\'emontrer la trivialit\'e, il suffit de d\'emontrer la trivialit\'e de l'invariant $\tilde{\pi}(\rho')$  
de $\SLb_1(A)$.  Comme annonc\'e, on utilise le morphisme $\theta$ de Merkurjev \eqref{eq:invarmerk} pour ce but. 
On d\'emontre que $\theta(\tilde{\pi}(\rho'))=0$ 
en travaillant sur la d\'efinition des groupes $\tilde{A}^0$ et utilisant
des r\'esultats connus.
\end{enumerate}

\subsection{Objets de base} \label{sec:relevbase}

Avant de relever des invariants, on doit \^etre capable de relever les objets de
base d'une fa\c{c}on adapt\'ee.  Nous expliquons donc en bref le rel\`evement des
corps et des alg\`ebres simples centrales.

\subsubsection{Alg\`ebres simples centrales} \label{sec:relevasc}

Soient $k$ un corps et $A$ une $k$-alg\`ebre simple centrale.  Soit $R$ un anneau complet de valuation discr\`ete $v$ tel que $k$ soit le corps r\'esiduel et tel que  $K=\Frac(R)$
est de caract\'eristique 0 (par exemple, un anneau de Cohen $R$ avec son corps
de fractions -- voir \ref{sec:cohen}).  

Soient $P(R)$, respectivement $P(k)$, l'ensemble des classes d'isomorphismes des $R$-alg\`ebres
d'Azumaya, respectivement des $k$-alg\`ebres simples centrales.  Alors, l'application r\'esiduelle 
$P(R)\to P(k)$ qui associe \`a la classe d'une $R$-alg\`ebre d'Azumaya $B$ la classe de $B\otimes_R k$,
est bijective \cite[Thm. 6.1]{grothbrauer}.
Alors, il existe une $R$-alg\`ebre d'Azumaya \textit{relev\'ee} (i.e. tel que $B\otimes_R k\cong A$)
de $A$ de m\^eme indice et degr\'e, appelons-la $B$.  Alors, $B_K:=B\otimes_R K$ est une $K$-alg\`ebre simple
centrale.  La bijection $P(R)\to P(k)$ induit un isomorphisme $\Br(R)\cong \Br(k)$, et de plus il y a
une injection $\Br(R)\to \Br(K)$ \cite[Thm. 7.2]{ausgold}.  Donc en somme, on a une injection $\Br(k)\to \Br(K)$.

Heureusement, $\SKb_1(A)(k)$ et $\SKb_1(B_K)(K)$ sont isomorphes.  Ce r\'esultat est essentiellement
d\^u \`a Platonov pour des alg\`ebres \`a division.  Pour une $K$-alg\`ebre \`a division $D$, la valuation $v$ s'\'etend \`a une valuation $w=\frac{1}{m}v\circ \Nrd_{D/K}$ sur $D$ avec 
$\Nrd_{D/K}$ la norme r\'eduite de $D$ et $m>0$  le g\'en\'erateur de $v\circ \Nrd_{D/K}(D)\subset \mathbb{Z}$
\cite[Ch. XII, \S 2]{serrecorloc}. Soit $\Ocal_D$ l'anneau de valuation de $w$ avec id\'eal maximal $\Pcal_D$, alors on note 
$\Dbar=\Ocal_D/\Pcal_D$, la \textit{$k$-alg\`ebre \`a division r\'esiduelle} -- voir aussi \cite[\S 2]{wadsworth}.  
L'application r\'esiduelle $\Ocal_D\to \Dbar$ se restreint alors \`a un morphisme r\'esiduel 
$\SLb_1(D)(K)\to \SLb_1(\Dbar)(k)$, et Platonov a d\'emontr\'e la propri\'et\'e de rigidit\'e suivante.

\begin{thm}[{\cite[Prop. 3.4, Thm. 3.12, Cor. 3.13]{sk1niettriv}}] \label{thm:platsksl}
Soient $K$ un corps complet de valuation discr\`ete $v$ avec $k$ son corps r\'esiduel et $D$
une $K$-alg\`ebre \`a division.  Le morphisme r\'esiduel
\[ \SLb_1(D)(K) \to \SLb_1(\Dbar)(k) \]
est surjectif de noyau contenu dans $[D^\times,D^\times]$.  Il factorise donc
\`a travers $\SKb_1(D)$, et cette factorisation induit un isomorphisme
\[ \SKb_1(D)(K) \cong \SKb_1(\Dbar)(k). \]
\end{thm}

Le but est  d'en d\'eduire l'isomorphisme entre  $\SKb_1(A)(k)$ et $\SKb_1(B_K)(K)$.  Bien \'evidemment,
on utilise les th\'eor\`emes consacr\'es de Wedderburn et de Morita \cite[Thm. 2.1.3, Lem. 2.8.6]{gilleszam}.
\begin{corr} \label{corr:sk1iso}
Soient $A,B,k,R$ et $K$ comme dessus, alors 
\[ \SKb_1(A)(k) \cong \SKb_1(B_K)(K). \]
\end{corr}

\begin{proof} 
Selon le th\'eor\`eme de Wedderburn, $B_K=M_m(D)$ pour une $K$-alg\`ebre \`a division $D$ et $m>0$.  D\^u \`a l'injectivit\'e
de $\Br(R) \to \Br(K)$, on sait que $M_m(\Ocal_D)$ est Brauer-\'equivalent \`a $B$.  Donc 
de nouveau d\^u \`a Wedderburn, $A=M_{m}(\Dbar)$.  Alors, le Th\'eor\`eme \ref{thm:platsksl}
et le th\'eor\`eme de Morita garantissent que
\[ \SKb_1(B_K)(K) \cong \SKb_1(D)(K) \cong \SKb_1(\Dbar)(k) \cong \SKb_1(A)(k). \]
\end{proof}

\begin{rem}
Il faut noter aussi que cet isomorphisme est fonctoriel dans le sens suivant.  Soit $K'$  une extension de corps de $K$ qui est \'egalement  complet et de valuation discr\`ete $v'$ tel que $v'$ prolonge $v$ et  soit $k'$ le corps r\'esiduel de $K'$.  Alors, l'isomorphisme  commute avec l'extension de base de $K$ \`a $K'$ et de $k$ \`a $k'$.
On n'a quand m\^eme pas d'\'equivalence de foncteurs, puisque il n'y a pas de bijection entre les extensions de $k$ et celles de $K$.
\end{rem}

\subsubsection{Anneaux de Cohen} \label{sec:cohen}

Les anneaux de Cohen -- parfois aussi appel\'es des \textit{$p$-anneaux}-- fournissent 
une m\'ethode assez canonique de relever des corps de caract\'eristique positive \`a des
anneaux  de caract\'eristique 0.
Nous commen\c{c}ons par la d\'efinition d'un anneau de Cohen.

\begin{defin}
Un \textit{anneau de Cohen} est un anneau complet de valuation discr\`ete  dont le
corps r\'esiduel est de caract\'eristique $p>0$, et dont l'id\'eal maximal est engendr\'e
par $p$.  Pour un premier $p$ fix\'e, on parle des $p$-anneaux de Cohen.
\end{defin}

Schoeller donne une construction explicite de ces anneaux \cite[\S 3]{schoeller}.  Ils sont
des sous-anneaux des anneaux de Witt. Dans le cas o\`u le corps r\'esiduel est parfait,
ils sont exactement les anneaux de Witt.  En g\'en\'eral, l'anneau de Cohen contient l'anneau de Witt de son sous-corps parfait
maximal.  \`A noter aussi que les anneaux de Cohen sont de caract\'eristique 0.
Nous rappelons le r\'esultat fondamental sur ces anneaux.

\begin{thm}[{\cite{cohen}, voir aussi \cite[Thm. 19.8.6]{ega4}}]\ \label{thm:cohen}
\begin{enumerate}[(i)]
\item \label{thm:coheni} Soient $W$ un anneau de Cohen, $C$ un anneau local noeth\'erien complet et $I$ un id\'eal
de $C$ distinct de $C$.  Alors tout homomorphisme local $u:W\to C/I$ se factorise en 
$W \overset{v}{\to} C \to C/I$, o\`u $v$ est un homomorphisme local.
\item Soit $k$ un corps de caract\'eristique $p>0$.  Il existe un anneau de Cohen $W$ dont le corps r\'esiduel est isomorphe
\`a $k$.   Si $W'$ est un second anneau de Cohen, $k'$ son corps r\'esiduel, tout isomorphisme 
$u:k\to k'$ provient par passage au quotient d'un isomorphisme $v:W\to W'$.
\end{enumerate}
\end{thm}

\begin{rem}
Remarquons que la propri\'et\'e \eqref{thm:coheni} induit que les anneaux de Cohen sont des objets
initiaux dans la cat\'egorie des anneaux locaux avec corps r\'esiduel fixe.
Ce th\'eor\`eme semble donc sugg\'erer qu'il existe une propri\'et\'e universelle des anneaux de Cohen.
Malheureusement, les morphismes induits ne sont pas uniques en g\'en\'eral.  Il ne le sont
que si $k$ est parfait (avec $k$ le corps r\'esiduel de $W$) \cite[Rem.  21.5.3]{ega4}.  Donc par  manque d'unicit\'e, on appelle cette propri\'et\'e universelle amput\'ee
une propri\'et\'e \textit{verselle}, comme le fait Serre \cite[\S 5]{cohinv}.
Heureusement cet handicap se gu\'erit au niveau de la cohomologie.
\end{rem}

\begin{corr} \label{corr:cohencoh}
Soient $W,W'$ des anneaux de Cohen tel que 
le corps r\'esiduel $k'$ de $W'$ est une extension de $k$, le corps r\'esiduel de $W$.  Notons $u:k\to k'$ l'inclusion.
Le Th\'eor\`eme \ref{thm:cohen} (\ref{thm:coheni}) fournit un homomorphisme local
$v:W\to W'$.   Soient $A$ une $k$-alg\`ebre simple centrale et $B$ la $W$-alg\`ebre d'Azumaya relev\'ee.
Soient $K=\Frac(W)$ et $K'=\Frac(W')$, alors $v$ d\'efinit
pour tous entiers $i,n\geq 0$ un homomorphisme des suites exactes scind\'ees 
\[
\xymatrix{
0 \ar[r] & 
H^{i+1}_{n,A}(k) \ar[r] \ar[d]^{u_\ast} & 
H^{i+1}_{n,B_K}(K) \ar[r]^{\partial^i} \ar[d]^{v_\ast} &
H^i_{n,A} (k) \ar[r] \ar[d]^{u_\ast} &
0 \\
0 \ar[r] & 
H^{i+1}_{n,A}(k') \ar[r] & 
H^{i+1}_{n,B_K}(K') \ar[r]^{\partial^i} &
H^{i}_{n,A} (k') \ar[r] &
0.
}
\]
De plus, $v_\ast$ ne d\'epend pas du choix de $v$.  Si $k=k'$, on a un isomorphisme canonique $H^{i+1}_{n,B_K}(K)\cong H^{i+1}_{n,B_K}(K')$.
\end{corr}

\begin{proof}
Soient $k,k'$ de caract\'eristique $p$, l'homomorphisme local $v$ envoie l'uniformisante $p\in W$
sur $p\in W'$. 
Donc, le diagramme et l'ind\'ependance du choix de $v$ suivent directement du scindage de la suite \eqref{eq:suitemod} par le cup-produit avec la classe de $p$.
Si $u$ est un isomorphisme, $v$ l'est aussi par le Th\'eor\`eme \ref{thm:cohen} \eqref{thm:coheni}, et on retrouve
un isomorphisme de suites exactes.
\end{proof}

Les anneaux de Cohen fournissent donc  au niveau de ces groupes de cohomologie une m\'ethode canonique pour passer de la caract\'eristique $p$ \`a la caract\'eristique 0,
et vice versa.  Pour faciliter la notation, introduisons une notion de triplets.

\begin{defin}
Soit $F$ un corps de valuation discr\`ete $v$,
alors on dit que $(F,\Ocal_v,\kappa(v))$ est un \textit{triplet de valuation}.  Un triplet de valuation 
$(K,R,k)$ o\`u $R$ est un $p$-anneau de Cohen (pour un premier $p>0$) est appel\'e un  
\textit{$p$-triplet de Cohen}.  Une \textit{extension de Cohen} (finie, resp. s\'eparable, resp. galoisienne) $(K',R',k')$ 
de $(K,R,k)$ est la donn\'ee d'un $p$-triplet de Cohen tel que $k'$ est une extension (finie, resp. s\'eparable, resp. galoisienne) de $k$.  
\end{defin}

\begin{rem}
Si cela ne pr\^ete pas \`a confusion, on parle librement d'un triplet de Cohen sans mentionner la caract\'eristique.
\'Etant donn\'e un corps $k$ de $\car(k)=p>0$, le Th\'eor\`eme \ref{thm:cohen} fournit un triplet de Cohen $(K,R,k)$ 
(non unique) \textit{associ\'e \`a $k$}.  De plus, si $(K',R',k')$ est une extension de Cohen 
(finie, resp. s\'eparable, resp. galoisienne) de $(K,R,k)$, le Th\'eor\`eme \ref{thm:cohen} induit que 
$K'$ est une extension (finie, resp. non-ramifi\'ee, resp. galoisienne) de $K$ -- voir aussi \cite[\S III.5]{serrecorloc}.
Si $(K,R,k)$ est un triplet de Cohen, $F$ un $R$-corps et $(F,\Ocal_v,\kappa(v))$ un triplet de valuation, 
on dit que $(F,\Ocal_v,\kappa(v))$ est \textit{un triplet de valuation sur $R$}.  Dans ce cas, $\kappa(v)$ est aussi un
$R$-corps et il y a trois possibilit\'es: ou bien $F$ et $\kappa(v)$ sont des extensions de $K$, ou bien ils sont toutes les deux des extensions de $k$, ou bien $F$ est une extension de $K$ et $\kappa(v)$ est une extension
de $k$.
\end{rem}

\subsection{Le rel\`evement}

On a maintenant fait les pr\'eparations pour relever l'invariant de Suslin en caract\'eristique mod\'er\'ee.  

\begin{thm} \label{thm:modere}
Soient $k$ un corps de $\car(k)=p>0$ et $A$ une $k$-alg\`ebre simple centrale de $\ind_k(A)=n$ premier \`a $p$.  Soit $(K,R,k)$ un triplet de Cohen associ\'e \`a $k$,  
$B$ la $R$-alg\`ebre d'Azumaya relev\'ee et $\rho'\in \Inv^4(\SKb_1(B_K),\Hcal_{n,\Bcal_K}^\ast)$.
Alors, il existe un unique $\rho \in \Inv^4(\SKb_1(A),\Hcal_{n,\Acal}^\ast)$, que l'on appelle \textit{l'invariant sp\'ecialis\'e de $\rho'$},
tel que pour toute extension de Cohen $(K',R',k')$ de $(K,R,k)$
le diagramme suivant commute:
\begin{equation} \label{diag:inv}
\xymatrix{ 
\SKb_1(A)(k')  \ar[r]^{\ \ \rho_{k'}} & H^4_{n,A}(k') \ar[d] \\
\ar[u]_{\cong} \SKb_1(B_K)(K') \ar[r]_{\ \ \ \rho'_{K'}} 
  & H^4_{n,B_K} (K'). 
}
\end{equation}
\end{thm}

\begin{rem}
Les modules de cycles $\Hcal_{n,\Bcal_K}^\ast:=(H^{j}_{n,B_K})_{j\geq 2}$ de base $K$ et $\Hcal_{n,\Acal}^\ast:=(H^{j}_{n,A})_{j\geq 2}$ de base $k$ sont les modules de cycles \'evidents.  Il sont les modules de cycles restreints du module de cycles $\Hcal_{n,\Bcal_K}^\ast$
de base $R$ respectivement \`a $K$ et \`a $k$ (selon \S \ref{sec:def} \ref{sec:coexistence}).
Le morphisme $H^4_{n,A}(k') \to H^4_{n,B_K} (K')$
est donc l'injection de la suite exacte \eqref{eq:suitemod}.
\end{rem}

Commen\c{c}ons par le deuxi\`eme pas de la strat\'egie g\'en\'erale expliqu\'ee dans \ref{sec:strategie}.  On utilise  la proposition suivante qui est une donn\'ee importante.

\begin{prop}[Merkurjev {\cite[Lem. 4.8 et Prop. 4.9]{invalggroup}}] \label{prop:merkurjev}
Soient $k$ un corps et $\GB$ un $k$-groupe semi-simple simplement connexe, alors $ \tilde{A}^0(\GB,\Hcal_n^3) =0$ pour tout $n\in k^\times$.  En particulier (d'apr\`es \S \ref{sec:lieninv} \ref{sec:lienmerk}), $\Inv^3(\GB,\Hcal_n^\ast)=0$.
\end{prop}

Nous nous permettons d'en d\'eduire un r\'esultat auxiliaire \`a base de quelques arguments homologiques.

\begin{corr}\label{corr:invnulmod}
Soient $k$ un corps, $\GB$ un $k$-groupe semi-simple simplement connexe, $A$ une $k$-alg\`ebre simple centrale tel que $\ind_k(A)=n\in k^\times$, alors
$ \Inv^3(\GB,\Hcal_{n,\Acal}^\ast)=0$.
\end{corr}

\begin{proof}
Selon \S \ref{sec:lieninv} \ref{sec:lienmerk}, il suffit de d\'emontrer la trivialit\'e de 
$ \tilde{A}^0(\GB,\Hcal_{n,\Acal}^3)$.
Tout d'abord on consid\`ere le diagramme commutatif
\begin{equation} \label{diag:grand}
\xymatrix{
H^1(k, \mu_n) \ar[r] \ar[d]^{\cup [A]}
& H^1(k(\GB), \mu_n) \ar[r]^-{\partial^1} \ar[d]^{\cup [A_{k(\GB)}]}
& \bigoplus_{x\in \GB^{(1)}} H^0(k(x), \Zb/n\Zb)  \ar[d]^{\oplus_{x\in \GB^{(1)}} \cup [A_{k(x)}]}
\\
 H^3_n(k) \ar[r] \ar[d]
& H^3_n(k(\GB)) \ar[r]^-{\partial^{3}} \ar[d]
& \bigoplus_{x\in \GB^{(1)}} H^2_n(k(x)) \ar[d]
\\ H^3_{n,A}(k) \ar[r] 
& H^3_{n,A}(k(\GB)) \ar[r]^-{\partial^{3}_{A}}
& \bigoplus_{x\in \GB^{(1)}} H^2_{n,A}(k(x)) ,
}
\end{equation}
dont les lignes sont des complexes de cha\^ines, et dont la ligne centrale est exacte vu la Proposition \ref{prop:merkurjev} et  l'isomorphisme \eqref{eq:invarmerk}.
Il suffit de d\'emontrer l'exactitude de la ligne d'en bas.  
Le th\'eorie de Kummer et les propri\'et\'es des r\'esidus \cite[Rem. 6.2]{cohinv}
indiquent que $\partial^1$, qui est une somme des r\'esidus, est en r\'ealit\'e le morphisme diviseur:
\[ k(\GB)^\times/(k(\GB)^\times)^n \to \bigoplus_{x\in \GB^{(1)}} \Zb/n\Zb = \Div(\GB)/n\Div(\GB): \overline{f} \mapsto \overline{\divi(f)}.  \] 
Il est surjectif puisque $\Pic(\GB)=0$ \cite[Lem. 6.9]{sansuc}.
 Se basant sur la surjectivit\'e 
de $\partial^1$, on d\'emontre alors l'exactitude du complexe en bas avec une chasse au diagramme.
\end{proof}

\begin{proof}[D\'emonstration du Th\'eor\`eme \ref{thm:modere}.]
 Soit $\rho'\in \Inv^4(\SKb_1(B_K),\Hcal_{n,\Bcal_K}^\ast)$. 
Afin de construire $\rho \in \Inv^4(\SKb_1(A),\Hcal_{n,\Acal}^\ast)$,
il faut commencer par d\'efinir $\rho_{k'}:\SKb_1(A)(k') \to H^4_{n,A}(k')$ pour toute extension
de corps $k'$ de $k$,
et puis prouver la fonctorialit\'e en le corps.  
Soit donc $(K',R',k')$ une extension de Cohen associ\'ee \`a $k'$.  
On a donc $\rho'_{K'}: \SKb_1(B_{K})(K') \to H^{4}_{n,B_K}(K')$.  Soit $\pi$ l'isomorphisme $\SKb_1(B_{K'})(K')\cong \SKb_1(A)(k)$,
alors on d\'efinit \[ \tilde{\rho}_{k'} := \partial^4 \circ \rho_{K'} \circ \pi^{-1}: \SKb_1(A)(k') \to H^{3}_{n,A}(k').\]  
\`A noter aussi que ce proc\'ed\'e ne d\'epend pas du
choix de l'extension de Cohen.  Car si $(K'',R'',k')$ est une autre extension de Cohen associ\'ee \`a $k'$, le Corollaire \ref{corr:cohencoh} garantit un isomorphisme des suites exactes scind\'ees de type (\ref{eq:suitemod}) avec l'identit\'e au facteurs $H^{4}_{n,A}(k')$ et $H^{3}_{n,A}(k')$.  De plus $\rho'_{K'}$ et $\pi$ sont fonctoriels pour de telles extensions de corps, 
donc ce proc\'ed\'e construit bien un invariant $\tilde{\rho}\in \Inv^3(\SKb_1(A),\Hcal^\ast_{n,\Acal})$.  

Le Corollaire \ref{corr:invnulmod} et la Propri\'et\'e \ref{prop:sksl} disent alors que
$\tilde{\rho}=0$.  C'est-\`a-dire si $a\in \SKb_1(A)(k')$, on sait que 
$\rho'_{k'} (a)$ provient d'un \'el\'ement unique de $H^4_{n,A}(k')$ (par la suite exacte \eqref{eq:suitemod}).  
Donc on a un morphisme $\rho_{k'}:\SKb_1(A)(k')\to H^4_{n,A}(k')$.  
De nouveau puisque, la suite \eqref{eq:suitemod}
est fonctorielle, et le choix de l'anneau de Cohen n'a pas d'influence sur la d\'efinition, ceci d\'efinit un invariant $\rho\in \Inv^4(\SKb_1(A),\Hcal_{n,\Acal}^\ast)$.  

Le diagramme commutatif \eqref{diag:inv} suit de la construction, et l'unicit\'e suit imm\'ediatement
de l'injectivit\'e de $H^4_{n,A}(k') \to H^4_{n,B_K} (K')$.
\end{proof}

\begin{rem}
Pour un corps $k$ de $\car(k)=p>0$ et $A$ une $k$-alg\`ebre simple centrale de $\ind_k(A)\in k^\times$, Suslin a d\'ej\`a defini un invariant $\rho_{\text{Sus},A}$.  N\'eanmoins, si $(K,R,k)$ est un $p$-triplet de Cohen, $B$ la $R$-alg\`ebre d'Azumaya relev\'ee et
$\rho_{\text{Sus},B_K}$ l'invariant de Suslin de $B_K$, l'invariant sp\'ecialis\'e de $B_K$ est quand m\^eme le 
m\^eme invariant que $\rho_{\text{Sus},A}$.  En effet, on peut v\'erifier 
que l'invariant de Suslin satisfait cette propri\'et\'e de rel\`evement (i.e. il existe un diagramme comme 
\eqref{diag:inv} avec $\rho=\rho_{\text{Sus},A}$ et $\rho'=\rho_{\text{Sus},B_K}$).
\end{rem}

\section{Le rel\`evement, le cas sauvage} \label{sec:sauvage}

Soient $k$ un corps de caract\'eristique $p>0$ et $A$ une $k$-alg\`ebre simple centrale d'indice \'eventuellement divisible par $p$.  
On se trouve maintenant dans un monde inconnu, puisque le module de cycles $\Hcal^\ast_{n,\Acal}$
n'est plus adapt\'e \`a nos buts.  En effet, parce que $\mu_{p^n}(k_s)$ est trivial, 
les groupes de cohomologie galoisienne $H^{j+1}(k,\mu_{p^n}^{\otimes j})$ sont triviaux, et 
la suite exacte de Kummer \eqref{eq:kummer} ne donne plus d'isomorphisme de $H^2(k,\mu_{p^n})$
avec $\phantom{ }_{p^n}\Br(k)$.   

Dans cette section, on d\'ecrit de nouveaux groupes de cohomologie (introduits par Kato \cite{katogalcoh}) qui fournissent
aussi un isomorphisme avec $\phantom{ }_{p^n}\Br(k)$ (dont on a besoin afin de pouvoir d\'efinir des modules de cycles relatifs comme dans \S \ref{sec:cohomdef} \ref{sec:relatifmod}). Ceci donne assez d'ingr\'edients pour
effectuer le rel\`evement.  Dans une premi\`ere \'etape, nous effectuons le travail pour une alg\`ebre simple
centrale d'indice $p^n$.  Le rel\`evement pour le cas g\'en\'eral peut en \^etre d\'eduit utilisant
la d\'ecomposition d'une alg\`ebre \`a division en parties premi\`eres (\S \ref{sec:relevgen}). 

\subsection{Groupes de cohomologie} \label{sec:sauvagecoh}

On utilise des groupes de cohomologie des \textit{diff\'erentielles logarithmiques} de Kato (ibid.).  
Ils sont pourvus d'une suite exacte analogue \`a la suite \eqref{eq:suitemod} dont on
a besoin afin d'\'etablir le rel\`evement.  
Soient donc dans toute cette section $(K,R,k)$ un $p$-triplet de Cohen et $F$ un $R$-corps.  
Rappelons la notion de diff\'erentielles logarithmiques et la d\'efinition, ainsi que quelques propri\'etes des groupes $H^{q+1}_{p^n}(k)$ 
(pour des entiers $n,q\geq 0$)\footnote{L'indice $q+1$ est de nouveau d\^u \`a la tradition 
facilitant la notation de quelques r\'esultats.}.

\subsubsection{Diff\'erentielles logarithmiques} \label{sec:difflog}  
La d\'efinition de $H^{q+1}_{p^n}(k)$ est le plus claire pour $n=1$ et explique la terminologie.
Soient $\Omega_{k}^q:=\bigwedge \Omega_{k/\mathbb{Z}}^1$ et $d: \Omega_{k}^{q-1} \to \Omega_{k}^{q}$ la d\'erivation ext\'erieure 
usuelle (pour $q=0$, on pose $d=0$).
Alors, $H^{q+1}_p(k)$ est d\'efini comme le conoyau du \textit{morphisme de Cartier}
\[ F-1:\Omega_{k}^q \to \Omega_{k}^q/d\Omega_{k}^{q-1},\ \  \text{ defini par }\ \ x \frac{dy_1}{y_1} \wedge \ldots \wedge 
 \frac{dy_q}{y_q} \mapsto (x^p-x) \frac{dy_1}{y_1} \wedge \ldots \wedge 
 \frac{dy_q}{y_q} \mod d\Omega_{k}^{q-1},
\]
avec $x\in k$, $y_1,\ldots,y_q \in k^\times$ et $F(x)=x^p$ \cite[Ch. 2, \S 6]{cartier}.  Le noyau est not\'e traditionellement  $\nu_1(q)_k$.

\subsubsection{G\'en\'eralisation} \label{sec:diffloggen} On peut g\'en\'eraliser la d\'efinition de $H^{q+1}_{p}(k)$ 
\`a une d\'efinition de $H^{q+1}_{p^n}(k)$.
  Soit 
\begin{equation*}
D_{p^n}^q(k)= W_n(k)\otimes \underset{q \text{ fois}}{\underbrace{k^\times \otimes \ldots \otimes k^\times}}, 
\end{equation*}
avec $W_n(k)$ le groupe des $p$-vecteurs de Witt de longueur $n$ sur $k$.  Soit
$J'_{q}(k)$ le sous-groupe de $D_{p^n}^q(k)$ engendr\'e par tous les \'el\'ements de la forme
\vspace{3mm}
\begin{enumerate}[(i)]
 \item $w\otimes b_1 \otimes \ldots \otimes b_{q}$, satisfiant $b_i=b_j$ pour $1\leq i < j \leq q$. \label{it:J1}
\end{enumerate}
\vspace{3mm}
Le groupe $C_{p^n}^q(k)=D_{p^n}^q(k)/J'_q(k)$ est alors muni d'une d\'erivation $d:C_{p^{n}}^{q-1} (k) \to C_{p^{n}}^{q} (k)$: pour $a\in k$, $b_2,\ldots, b_q\in k^\times$ est $q>0$  d\'efinie par
\[ (0,\ldots,0,a,0,\ldots,0) \otimes b_2 \otimes \otimes \ldots b_{q} \mapsto 
(0,\ldots,0,a,0,\ldots,0)\otimes a \otimes b_2 \otimes \ldots \otimes b_{q}. 
 \]
Pour $q=0$, on pose de nouveau $d=0$. Le groupe de cohomologie $H^{q+1}_{p^n}(k)$ est  alors
d\'efini comme le conoyau du morphisme de Cartier 
\begin{eqnarray*} F-1: \ C_{p^n}^q(k) &\to& C_{p^n}^q(k)/dC_{p^{n}}^{q-1} (k), \quad \text{ d\'efini par } \\  w\otimes b_1 \otimes \ldots \otimes b_{q} &\mapsto &(w^{(p)}-w) \otimes b_1 \otimes \ldots \otimes b_{q}.
\end{eqnarray*}
Ici, $F(w)=w^{(p)}=(a_1^p,\ldots ,a_{n}^p)$ si $w=(a_1,\ldots,a_{n})$.  
On note $\nu_n(q)_k$ le noyau du morphisme de Cartier.
Alors, $H^{q+1}_{p^n}(k)\cong D_{p^n}^q(k)/J_q(k)$ o\`u $J_q(k)$ \cite[Proof of Prop. 2]{katogalcoh} est le sous-groupe de $D_{p^n}^q(k)$ engendr\'e par
les \'el\'ements de la forme \eqref{it:J1} et 
\vspace{3mm}
\begin{enumerate}[(i)] \setcounter{enumi}{1}
 \item $(0,\ldots,0,a,0,\ldots,0)\otimes a \otimes b_2 \otimes \ldots \otimes b_{q}$, \label{it:J2}
\item $(w^{(p)}-w) \otimes b_1 \otimes \ldots \otimes b_{q}$. \label{it:J3}
\end{enumerate}
\vspace{3mm}
De plus, on pose $H_{p^n}^{q+1}(k)=0$ pour $q<0$.  Cette construction est donc bien une g\'en\'eralisation de la d\'efinition
pour $n=1$ en termes de diff\'erentielles logarithmiques par l'identification 
\[ x \frac{dy_1}{y_1} \wedge \ldots \wedge 
 \frac{dy_q}{y_q} \quad \leftrightarrow \quad  x \otimes y_1 \otimes \ldots \otimes y_q,\]
avec $x\in k$ et $y_i\in k^\times$.  \`A noter que l'antisym\'etrie vaut pour cette g\'en\'eralisation, puisque $w\otimes b_1 b_2 \otimes b_1 b_2 \otimes \ldots b_q = 0$ ($w\in W_n(k), b_i\in k^\times$).

Soit  $\dlog:k_s^{\times}\to \nu_n(1)_{k_s}:a\mapsto 1 \otimes a$.  
La suite exacte longue associ\'ee \`a la suite exacte 
\begin{equation} \label{eq:suitedlog} 
\xymatrix{1 \ar[r]& k_s^\times \ar[r]^{p^n} & k_s^\times \ar[r]^{\dlog \quad} & \nu_n(1)_{k_s} \ar[r] & 1} 
\end{equation}
des $\Gamma_k$-modules induit un isomorphisme sur la partie de $p^n$-torsion du groupe de Brauer:
$ H^1(k,\nu_n(1)_{k_s}) \cong \phantom{ }_{p^n}\Br(k)$ . 
L'exactitude de \eqref{eq:suitedlog} suit d'un calcul avec des vecteurs de Witt et des produits tensoriels \cite[Ch. 2, Prop. 8]{cartier}.
De plus, on a une suite exacte: 
\[ \xymatrix{ 0 \ar[r] & \nu_n(q)_{k_s} \ar[r] & C_{p^n}^q(k_s) \ar[r]^{F-1\qquad \ \ } & C_{p^n}^q(k_s)/dC_{p^n}^{q-1}(k_s) \ar[r] & 0.  }\]
La surjectivit\'e de $F-1$ suit du Th\'eor\`eme \ref{thm:katoizh} qui d\'emontre que $H^{q+1}_{p^n}(k_s)=0$ pour chaque $q\geq 0$ et $n>0$.
Puisque $C_{p^n}^q(k_s)$ est un $k_s$-espace
vectoriel tel que $C_{p^n}^q(k_s)^\Gamma=C_{p^n}^q(k)$, on en d\'eduit donc par le th\'eor\`eme de Hilbert 90 additif l'isomorphisme
\begin{equation} \label{eq:isonu}
H^1(k,\nu_n(q)_{k_s}) \cong H_{p^n}^{q+1}(k).
\end{equation}
Comme dans le cas mod\'er\'e on a donc
\begin{equation} \label{eq:isobrauersauv}
 H^2_{p^n}(k) \cong \phantom{ }_{p^n}\Br(k).
\end{equation}

\subsubsection{Suite exacte de Kato} \label{sec:suitekato}
Afin de pouvoir \'etablir le m\^eme canevas que dans le cas mod\'er\'e, il faut qu'on ait un module de cycles et
on aimerait aussi avoir une suite exacte comme \eqref{eq:suitemod}.   Malheureusement, la th\'eorie
est plus complexe, mais on a quand m\^eme le r\'esultat suivant.

\begin{thm}[Kato \cite{katogalcoh}, Izhboldin \cite{izboldhin}] \label{thm:katoizh}
Soit $F$ un $R$-corps complet de valuation discr\`ete $v$ avec triplet de valuation $(F,\Ocal_v,\kappa(v))$.  Soient $F_{\text{nr}}$ l'extension
maximale non-ramifi\'ee de $F$ et 
\[ H^{q+1}_{p^n,\text{nr}}(F) := \ker (H^{q+1}_{p^n}(F) \to H^{q+1}_{p^n}(F_{\text{nr}})). \]
Alors, on a une suite exacte scind\'ee
\begin{equation}  \label{eq:suitesauv}
0 \to H^{q+1}_{p^n} (\kappa(v)) \to H^{q+1}_{p^n,\text{nr}} (F) \to H^{q}_{p^n} (\kappa(v)) \to 0.
\end{equation}
\end{thm}

\begin{rem} \label{rem:ressauv}
Expliquons le scindage et les morphismes en jeu. Selon les caract\'eristiques de $F$ et $\kappa(v)$, on discute  trois situations.
\begin{itemize}
\item Dans le cas d'\textit{in\'egale caract\'eristique} ($\car(F)=0$ et $\car(\kappa(v))=p$), le scindage est obtenu par des morphismes
d\^us \`a Kato \cite[Proof of Prop. 2]{katogalcoh}.  D'une part, on a une inclusion
 $i^\ast:H^{q+1}_{p^n} (\kappa(v))  \to  H^{q+1}_{p^n,\text{nr}} (F)$ de degr\'e 0, d\'efinie par
\[ w \otimes \bar{b}_1 \otimes \ldots \otimes \bar{b}_{q} \mod J_q(\kappa(v))  \mapsto 
i(w) \cup h^{q}_{p^n,F}(b_1,\ldots, b_{q}) \] 
D'autre part, on a une inclusion
$\psi: H^{q}_{p^n} (\kappa(v))  \to  H^{q+1}_{p^n,\text{nr}} (F)$ de degr\'e 1, d\'efinie par
\[ w \otimes \bar{b}_2 \otimes \ldots \otimes \bar{b}_{q} \mod J_{q-1}(\kappa(v)) \mapsto 
i(w) \cup h^{q}_{p^n,F}(\pi,b_2,\ldots, b_{q}). \]
Ici, $w\in W_n(\kappa(v))$, $\pi$ est une uniformisante fixe de $F$, $b_i\in \Ocal_v^\times$, $h ^{q}_{p^n,F}$ est le
symbole galoisien \eqref{eq:symbgal} et $i$ est l'homomorphisme canonique 
\[ W_n(\kappa(v))/\{ w^{(p)}-w | w \in W_n(\kappa(v))\} \overset{\varphi}{\cong} H^1(\kappa(v),\mathbb{Z}/p^n\mathbb{Z}) \hookrightarrow H^1_{p^n}(F).\]
La derni\`ere injection ci-dessus est celle de la suite \eqref{eq:suitemod}, et l'isomorphisme $\varphi$ vient du th\'eor\`eme 90 d'Hilbert additif
appliqu\'e \`a la suite exacte de cohomologie associ\'ee \`a  la suite exacte de Witt \cite[\S 5]{witt}:
\[ \xymatrix{ 0 \ar[r] & \mathbb{Z}/p^n\mathbb{Z} \ar[r] & W_n(\kappa(v)_s) \ar[r]^{x^{(p)}-x} & W_n(\kappa(v)_s) \ar[r] & 0. }\]
Kato d\'emontre donc que $i^\ast \oplus \psi$ donne un isomorphisme
\[ H^{q+1}_{p^n} (\kappa(v)) \oplus H^{q}_{p^n} (\kappa(v)) \cong H^{q+1}_{p^n,\text{nr}} (F). \]
\item Le cas d'\textit{\'egale caract\'eristique $0$} ($\car (F)=\car (\kappa(v))=0$) est comme le cas mod\'er\'e.  En effet,
 $H^{q+1}_{p^n,\text{nr}}(F)=H^{q+1}_{p^n}(F)$, puisque selon la suite scind\'ee \eqref{eq:sesrv4}, on a \[ H^{q+1}_{p^n}(F_{\text{nr}})\cong H^{q+1}_{p^n}(\kappa(v)_s)\oplus H^{q+1}_{p^n}(\kappa(v)_s)=0.\]

\item \label{it:katoizb} Le cas d'\textit{\'egale caract\'eristique $p$} ($\car (F)=\car (\kappa(v))=p$) est d\'ecrit par 
Izhboldin \cite[Prop. 6.8]{izboldhin}.  Le morphisme $i^\ast : H^{q+1}_{p^n} (\kappa(v)) \to H^{q+1}_{p^n,\text{nr}} (F)$
est d\'efini par 
\[ \bar{w} \otimes \bar{b}_1 \otimes \ldots \otimes \bar{b}_{q} \mod J_q(\kappa(v))  \mapsto 
w \otimes b_1 \otimes \ldots \otimes b_{q} \mod J_q(F). \]
D'autre part il y a un morphisme 
$\psi : H^{q}_{p^n} (\kappa(v)) \to H^{q+1}_{p^n,\text{nr}} (F)$, d\'efini par 
\[ \bar{w} \otimes \bar{b}_2 \otimes \ldots \otimes \bar{b}_{q} \mod J_{q-1}(\kappa(v))  \mapsto 
w \otimes \pi \otimes b_2 \otimes \ldots \otimes b_{q} \mod J_q(F), \]
o\`u $\pi$ est de nouveau une uniformisante fixe de $F$, $b_i\in \Ocal_v^\times$, $w=(a_1,\ldots,a_n)\in W_n(\Ocal_v)$ et $\bar{w}=(\bar{a}_1,\ldots,\bar{a}_n)$ son
r\'esidu dans $W_n(\kappa(v))$.  Izhboldin d\'emontre que $i^\ast \oplus \psi$ induit un scindage de $H^{q+1}_{p^n,\text{nr}} (F)$.
\end{itemize}
\end{rem}

\subsubsection{D\'efinition du $R$-module de cycles $\Hcal_{p^n,L}^\ast$} \label{sec:sauvdef} Maintenant, on a assez d'information afin de d\'efinir notre module de cycles envisag\'e.
\begin{defin} \label{def:modcycsauv}
Soient $(K,R,k)$ un $p$-triplet de Cohen avec une extension de Cohen finie, galoisienne $(L,S,\Lbar)$ . 
Alors pour tout entier $n>0$, on d\'efinit $\Hcal_{p^n,L}^\ast=(\Hcal_{p^n,L}^{i}(F))_{i> 0}$ le module de
cycles de base $R$ avec 
\[ \Hcal_{p^n,L}^{j+1}(F)=H_{p^n,L}^{j+1}(F):=
\begin{cases}
  \ker (H^{j+1}_{p^n}(F) \to H^{j+1}_{p^n}(F\otimes_K L)) & \text{ si } F\in \Kcorps, \\
  \ker (H^{j+1}_{p^n}(F) \to H^{j+1}_{p^n}(F\otimes_k \Lbar)) & \text{ si } F\in \kcorps. 
 \end{cases}
\]
\end{defin}

\begin{rem}
\`A noter que pour  $F\in \Kcorps$ les groupes de cohomologie sont des groupes de cohomologie galoisienne
usuelles.
\end{rem}

\begin{rem}
Si $(L_1,S_1,\Lbar_1)$ et $(L_2,S_2,\Lbar_2)$ sont deux extensions de Cohen finies galosiennes de $(K,R,k)$, 
il existe une extension de Cohen finie galoisienne $(L,S,\Lbar)$ de $(K,R,k)$ tel qu'elle est
 une extension de Cohen commune de $(L_1,S_1,\Lbar_1)$ et $(L_2,S_2,\Lbar_2)$. 
Alors, ils existent des injections $\Hcal_{n,L_1}^{\ast}\to \Hcal_{n,L}^{\ast}$ et $\Hcal_{n,L_2}^{\ast}\to \Hcal_{p^n,L}^{\ast}$.  Le  choix de $L$ n'a donc pas  grande importance.
\end{rem}

Il faut d\'emontrer que cette d\'efinition d\'efinit bien un module de cycles. 
 Pour cela, on a 
tout d'abord besoin des quatre donn\'ees D1-D4 (voir \S \ref{sec:def} \ref{sec:defgenmodcyc}).  Les donn\'ees D1, D2 et D3
se d\'eroulent seulement en \'egales caract\'eristiques.  Pour la donn\'ee D4, on peut rencontrer des 
caract\'eristiques mixtes.

Par cette remarque, la fonctorialit\'e (D1) suit imm\'ediatement. 
Soit $E$ une
extension finie de $F$ de trace $\Tr_{E/F}$.  Alors,  
$E\otimes_F C_{p^n}^q(F)\cong C_{p^n}^q(E)$, et on d\'efinit une trace de $C_{p^n}^q(E)$ par la composition
\[ C_{p^n}^q(E) \cong  E\otimes_F C_{p^n}^q(F) \overset{\Tr_{E/F} \otimes \id}{\verylongrightarrow} 
F \otimes_F C_{p^n}^q(F) \cong C_{p^n}^q(F).
\]
Celle-ci s'\'etend alors \`a une d\'efintion pour la r\'eciprocit\'e (D2) aux groupes de cohomologie $H_{p^n,L}^{q+1}(F)$.

\subsubsection{Structure de $K_m(F)$-module (D3)} \label{sec:sauvkth}

Si $\car(F)=0$ (i.e. $F$ est une extension de $K$), la structure de $K_m(F)$-module 
est d\'efinie comme dans le cas mod\'er\'e.  Si $\car(F)=p$ (i.e. $F$ est une extension de $k$),
cette structure est inspir\'ee par le \textit{symbole diff\'erentiel} au lieu 
du symbole galoisien.  
Pour tout $m\geq 1$,
\[ \rho_F^m: K_m(F) \to \Omega_F^m, \quad \text{ d\'efini par }\quad  \{ x_1,\ldots,x_m\} \mapsto \frac{dx_1}{x_1}\wedge \ldots \wedge \frac{dx_m}{x_m},
\]
est un homomorphisme. En effet, $d(ab)=bd(a)+ad(b)$ implique $\frac{d(ab)}{ab}=\frac{da}{a}+\frac{db}{b}$,
et si $a+b=1$, on a $\frac{da}{a}\wedge \frac{db}{b}=0$, puisque $da+db=0$ ($a,b\in k^\times$).
Alors, $\rho_F^m$ induit une application $K_m(F)/pK_m(F) \to \Omega_F^m$ parce que $\car(F)=p$ (et donc $dx^p=0$).  
De plus, l'image est contenue dans $\nu_1(m)_F$.  Le symbole diff\'erentiel
est $h^m_{1,F}:K_m(F)/pK_m(F) \to \nu_1(m)_F$.  Le Th\'eor\`eme de Bloch-Kato-Gabber d\'emontre que c'est
un isomorphisme \cite[Thm. 2.1]{blochkato}.  

Inspir\'e par la d\'efinition du symbole diff\'erentiel, on peut proposer la structure de
$K_m(F)$-module comme  suit:
\begin{eqnarray*}
 \rho_{p^n,F}^m:K_m(F) \times H_{p^n}^{q+1}(K) &\to& H_{p^n}^{q+m+1}(F): \\ (\{ x_1, \ldots, x_m \} , w \otimes b_1 \otimes \ldots \otimes b_{q} ) &\mapsto& w  \otimes x_1 \otimes \ldots \otimes x_m \otimes b_1 \otimes \ldots \otimes b_{q}.
\end{eqnarray*}
Les m\^emes arguments que ci-dessus garantissent qu'elle est bien d\'efinie.  Pour $a\in K_m(F)$ et $b\in H_{p^n}^{q+1}(F)$,
on note la multiplication scalaire par $a\cdot b:=\rho^m_{p^n,F}(a,b)$.
Cette structure se restreint \`a une structure de $K_m(F)$-module sur $(H_{p^n,L}^{q+1}(F))_{q\geq 0}$ pour $(L,S,\Lbar)$
comme dans la D\'efinition \ref{def:modcycsauv}.  En effet, si $b\in J_q(F\otimes \Lbar)$, on a $a\cdot b \in J_{q+m}(F\otimes  \Lbar)$ pour tout $a\in K_m(F)$.

\subsubsection{Le r\'esidu et une suite exacte} \label{sec:sauvres}
Puis, il faut l'existence d'un r\'esidu (la donn\'ee D4), et on veut aussi g\'en\'eraliser la suite exacte \eqref{eq:suitemod}.  

\begin{prop} \label{prop:suitesauv}
Soient $(K,R,k)$ un $p$-triplet de Cohen et $(L,S,\Lbar)$ une extension de Cohen finie galoisienne.  Pour tout $R$-triplet
de valuation $(F,\Ocal_v,\kappa(v))$ et pour tous entiers $n>0$ et $q\geq 0$, on a une suite exacte scind\'ee
\begin{equation} \label{eq:suitesauvl}
 0 \to  H_{p^n,L}^{q+1}(\kappa(v)) \to H_{p^n,L}^{q+1}(F) \to H_{p^n,L}^q(\kappa(v)) \to 0.
\end{equation}
\end{prop}

\begin{proof}
On a certainement deux versions de la suite \eqref{eq:suitesauv}:
\[
\xymatrix{ 
0 \ar[r] &  H^{q+1}_{p^n} (\kappa(v))\ar[d] \ar[r] &  H^{q+1}_{p^n,\text{nr}} (F) \ar[d] \ar[r] &  H^{q}_{p^n} (\kappa(v)) \ar[r] \ar[d]& 0 \\
0 \ar[r]  &  H^{q+1}_{p^n} (\kappa(v)\otimes \Lbar) \ar[r]  &  H^{q+1}_{p^n,\text{nr}} (F\otimes L) \ar[r]  &  H^{q}_{p^n} (\kappa(v)\otimes \Lbar) \ar[r]   & 0. \\
}
\]
Puisque les suites sont scind\'ees et les scindages respectent ce diagramme commutatif, la suite exacte scind\'ee suit du lemme du serpent.
\end{proof}

%

\begin{rem} \label{rem:ressauvnoncom}
Les r\'esidus pour un $R$-corps $F$ complet de valuation discr\`ete $v$ sont donc d\'efinis par cette suite.  Si $F$ n'est pas
complet, on prend $\hat{F}$ un complet\'e par rapport \`a $v$.  Le corps r\'esiduel  de $\hat{F}$ est \'egal au corps 
r\'esiduel $\kappa(v)$ de $F$, et alors le r\'esidu est d\'efini par la composition
\[ H^{i+1}_{p^n,L}(F) \to H^{i+1}_{p^n,L}(\hat{F}) \to H^{i}_{p^n,L}(\kappa(v)). \]
\end{rem}

On a donc introduit les quatre donn\'ees des modules de cycles, et en plus on a une suite qui nous permet de faire la m\^eme construction que dans le cas mod\'er\'e.  Pour la v\'erification d\'etaill\'ee 
des r\`egles de module de cycles, nous renvoyons \`a l'Appendice \ref{sec:append}.

\subsubsection{Version relative}

Comme dans \S \ref{sec:cohomdef} \ref{sec:relatifmod}, nous d\'efinissons de nouveau des modules de cycles relatifs \`a base
de l'isomorphisme \eqref{eq:isobrauersauv} et de l'action de la $K$-th\'eorie -- comparable avec  la pr\'esentation \eqref{eq:defaltmod} du module de cycle relatif moder\'e.

\begin{defin} \label{def:modcycsauvrel}
Soient $(K,R,k)$ un $p$-triplet de Cohen,
$A$ une $K$-alg\`ebre simple centrale de $\ind_K(A)=p^n$ et $\overline{A}$ la $k$-alg\`ebre
simple centrale r\'esiduelle.  Soit $(L,S,\Lbar)$ une extension de Cohen finie galoisienne de $(K,R,k)$ tel
que $\Lbar$ est  un corps de d\'ecomposition de $\overline{A}$.  Soient 
$F \in \rcorps$ et $[A_F]$ la classe de $A_F:=A\otimes_R F$ dans $\phantom{}_{p^n}\Br(F)$, alors
on d\'efinit le module de cycles $\Hcal_{p^n,L,\mathcal{A}}^{\ast}=(H_{p^n,L,A}^{j})_{j\geq 2}$ de base $R$ par
 \[ 
H_{p^n,L,A}^{j+1}(F) := H_{p^n,L}^{j+1}(F)/(K_{j-1}(F) \cdot [A_F]).
\]
\end{defin}

\begin{rem}
Le choix de $\Lbar$ est bien possible d\^u au th\'eor\`eme de Wedderburn qui donne une extension finie, s\'eparable $\Lbar'$ qui
d\'eploie $\overline{A}$.  On obtient $\Lbar$ en prenant une extension finie de $\Lbar'$ tel que l'extension $\Lbar/k$ est galoisienne.  Puis, on associe donc un triplet de Cohen $(L,S,\Lbar)$ \`a $\Lbar$.  

On peut m\^eme 
supposer que $\Lbar$ est une extension cyclique de $k$.
En effet, le th\'eor\`eme d'Albert \cite[Thm. 18]{albertcycl} dit que toute
$k$-alg\`ebre simple centrale de degr\'e $p^n$ est Brauer-\'equivalente \`a une $k$-alg\`ebre cyclique, et la d\'efinition de $\SKb_1(A)$ 
ne d\'epend pas du\break repr\'esentant de la classe de Brauer de $A$ (isomorphisme de Morita \cite[Lem. 2.8.6]{gilleszam}).

Le fait qu'on choisit $\Lbar$ comme corps de d\'ecomposition de $\overline{A}$ est afin de garantir que la multiplication scalaire 
se trouve dans $\Hcal_{p^n,L}^\ast$.  En effet, pour une extension $F$ de $k$ l'extension de base de $\overline{A}$ \`a $\Lbar$ rend $[\overline{A}_F]$ triviale dans le groupe de Brauer, et donc \'egalement le sous-groupe $K_{j-1}(F).[A_F]$.    
De plus, $L$ deploie $A$ parce que $\Br(\Lbar)\to \Br(L)$ est injective (\S \ref{sec:relevbase} \ref{sec:relevasc}).  
\end{rem}

Il faut donc v\'erifier que cette d\'efinition relative d\'efinit bien un module de cycles.  
On se fonde sur le fait qu'on dispose d'un module de cycles dans le cas absolu, et on
v\'erifie que les donn\'ees sont bien d\'efinies modulo les sous-groupes en jeu.
Les donn\'ees D1, D2 et D3 suivent plus au moins de la d\'efinition, et parce que les corps en jeu ont
toujours la m\^eme caract\'eristique.  La donn\'ee D4 pour le cas d'\'egale caract\'eristique (du corps de valuation
discr\`ete et de son corps r\'esiduel) suit aussi de la d\'efinition des modules de cycles et des r\'esidus de $\Hcal^\ast_{p^n,L}$.
Il suffit donc de v\'erifier la donn\'ee D4 pour le cas de caract\'eristique mixte.  De plus, on veut
g\'eneraliser la suite exacte \eqref{eq:suitesauvl}.

\begin{prop}
Avec les m\^emes donn\'ees de la D\'efinition \ref{def:modcycsauvrel}, on a pour tout triplet de valuation $(F,\Ocal_v,\kappa(v))$ 
sur $R$ une suite exacte 
\[ 0 \to  H_{p^n,L,A}^{q+1}(\kappa(v)) \to H_{p^n,L,A}^{q+1}(F) \to H_{p^n,L,A}^{q}(\kappa(v)) \to 0. \]
\end{prop}

\begin{proof}
Suite \`a la discussion pr\'ecedente, il suffit de d\'emontrer l'\'enonc\'e dans le cas de caract\'eristique mixte ($\car(F)=0$ et
$\car(\kappa(v))=p$).
Le but est donc de v\'erifier que la suite exacte \eqref{eq:suitesauvl} commute avec les inclusions dans un
diagramme commutatif (au signe pr\`es et pour $q\geq 2$)
\[ 
\xymatrix{ 
0 \ar[r] & 
H_{p^n,L}^{q+1}(\kappa(v)) \ar[r]^{i^\ast} &
H_{p^n,L}^{q+1}(F) \ar[r]^\partial &
H_{p^n,L}^{q}(\kappa(v)) \ar[r] &
0 
\\
0 \ar[r] &
K_{q-1}(\kappa(v))\cdot [\overline{A}_{\kappa(v)}] \ar@{^(->}[u] \ar@{.>}[r] &
K_{q-1}(F) \cdot [A_F] \ar@{^(->}[u] \ar@{.>}[r] &
K_{q-2}(\kappa(v))\cdot [\overline{A}_{\kappa(v)}] \ar@{^(->}[u] \ar[r] &
0.
}
\]
V\'erifions tout d'abord qu'au signe pr\`es le diagramme
\begin{equation} \label{eq:diabrres}
\xymatrix{
H^2_{p^n}(\kappa(v)) \ar[d]_{\cong} \ar[r]^{i^\ast}  & H^2_{p^n,\text{nr}}(F) \ar[d]^{\cong}   \\
_{p^{n}}\Br(\kappa(v)) \ar[r]^i & _{p^{n}}\Br_{\text{nr}}(F)  
}
\end{equation}
commute, o\`u $\Br_{\text{nr}}(F)=\ker(\Br(F)\to \Br(F_{\text{nr}}))$, $i^\ast$ le morphisme de la suite \eqref{eq:suitesauv} et $i$ l'injection de \S \ref{sec:relevbase} \ref{sec:relevasc}.  Alors, la v\'erification est un calcul imm\'ediat, faisons-le.  Soit $\bar{a}\otimes \bar{x}\in H^2_{p^n}(\kappa(v))$ avec
$a \in W_n(\Ocal_v)$ et $x\in \Ocal_v^\times$.  Alors, 
\[ i^\ast(\bar{a}\otimes \bar{x})=\left((\tau(y)/y)^{\sigma(b)-b}\right)_{\sigma,\tau}\in H^2_{p^n}(F)\]
avec $y^p=x$ et $a=b^p-b$ pour $y\in W_n(F_{\text{nr}})$ et $b\in F_{\text{nr}}^\times$.  Alors, l'image dans $\phantom{}_{p^n}H^2(F,F_s^\times)\cong \phantom{}_{p^n}\Br(F)$ est repr\'esent\'ee par la m\^eme expression.  Par ailleurs, l'image de $\bar{a}\otimes \bar{x}\in H^2_{p^n}(\kappa(v))$ dans $\phantom{}_{p^n}H^2(\kappa(v),\kappa(v)_{s}^\times)\cong \phantom{}_{p^n}\Br(\kappa(v))$ est $c:=\bigl((\sigma(\bar{y})/\bar{y})^{\tau(\bar{b})-\bar{b}}\bigr)_{\sigma,\tau}$.  Alors,
\[ i(c)=\left((\sigma(y)/y)^{\tau(b)-b}\right)_{\sigma,\tau}\in H^2_{p^n}(F).\]  
Parce que $i^\ast$ est d\'efini par un cup-produit, il est \'egal \`a $i^\ast(\bar{a}\otimes \bar{x})$
\`a signe pr\`es.

La restriction de \eqref{eq:diabrres}  aux sous-groupes donne le diagramme commutatif (au signe pr\`es)
\[
\xymatrix{
H^2_{p^n,L}(\kappa(v)) \ar[d]_{\cong} \ar[r]^{i^\ast}  & H^2_{p^n,L}(F) \ar[d]^{\cong}   \\
_{p^{n}}\Br(\Lbar \otimes_k \kappa(v)/\kappa(v)) \ar[r]^i & _{p^{n}}\Br(L\otimes_K F/F),  
}
\]
avec 
\begin{eqnarray*}
  \Br\left(L\otimes_K F/F\right)&=&\ker \bigl[ \Br(F)\to \Br(L\otimes_K F)\bigr] \quad \text{ et } \\
\Br\left(\Lbar \otimes_k \kappa(v)/\kappa(v)\right)&=&\ker\bigl[\Br(\kappa(v))\to \Br(\Lbar \otimes_k \kappa(v))\bigr].
\end{eqnarray*}

La d\'emonstration de ce th\'eor\`eme suit tout directement de ce fait parce que les morphismes $i^\ast, \partial$ et la retraction $\psi$ (voir \ref{sec:suitekato}) respectent la structure de modules
de $K$-th\'eorie, et parce que le signe dispara\^\i t au niveau des groupes de quotient.  
Donc,
\begin{eqnarray*} 
i^\ast\bigl(K_{q-1}(\kappa(v))\cdot [\overline{A}_{\kappa(v)}]\bigr) & =& 
i^\ast_K\bigl(K_{q-1}(\kappa(v))\bigr)\cdot i^\ast\bigl( [\overline{A}_{\kappa(v)}]\bigr) \subset K_{q-1}(F) \cdot [A_F], \\
\partial\bigl(K_{q-1}(F) \cdot [A_F]\bigr) & = & \partial_K\bigl(K_{q-1}(F)\bigr) \cdot [\overline{A}_{\kappa(v)}] = K_{q-2}(\kappa(v)) \cdot [\overline{A}_{\kappa(v)}] \quad \text{et} \\
\psi\bigl(K_{q-2}(\kappa(v))\cdot [\overline{A}_{\kappa(v)}]\bigr) & =& 
\psi_K\bigl(K_{q-2}(\kappa(v))\bigr)\cdot i^\ast\bigl( [\overline{A}_{\kappa(v)}]\bigr) \subset K_{q-1}(F) \cdot [A_F].
\end{eqnarray*}
Ici,
\begin{eqnarray*}
i^\ast_K: K_{q-1}(\kappa(v))\to K_{q-1}(F),\quad & \text{d\'efini par} & \quad \{ \bar{x}_1,\ldots, \bar{x}_{q-1} \} \mapsto \{ x_1,\ldots, x_{q-1} \},\\
\partial_K: K_{q-1}(F) \to K_{q-2}(\kappa (v)), \quad & \text{d\'efini par} & \quad 
\begin{cases} 
\{ x_1,\ldots, x_{q-1} \} \mapsto 0, & \\
\{ \pi, x_2,\ldots, x_{q-2} \} \mapsto \{ \bar{x}_2,\ldots, \bar{x}_{q-2} \}, & 
\end{cases} \\
\psi_K: K_{q-2}(\kappa(v))\to K_{q-1}(F),\quad & \text{d\'efini par}& \quad \{ \bar{x}_2,\ldots, \bar{x}_{q-2} \} \mapsto \{ \pi, x_2,\ldots, x_{q-2} \},
\end{eqnarray*}
pour $x_i\in \Ocal_v^\times$ et $\pi$ une uniformisante de $F$ pour la valuation $v$.
\end{proof}

Notons que cette suite exacte satisfait une propri\'et\'e analogue au Corollaire \ref{corr:cohencoh}, puisque dans ce cas les scindages sont  aussi
donn\'es par un choix d'uniformisante qui est canonique pour des anneaux de Cohen -- voir les d\'efinitions dans la Remarque \ref{rem:ressauv}.

\begin{corr}
Prenons les m\^emes donn\'ees de la D\'efinition \ref{def:modcycsauvrel} et $(F,\Ocal_v,\kappa(v))$ une extension de
Cohen de $(K,R,k)$.  Notons $u:k\to \kappa(v)$ l'inclusion.
Le Th\'eor\`eme \ref{thm:cohen} (\ref{thm:coheni}) fournit un homomorphisme local
$v:R\to \Ocal_v$, alors $v$ d\'efinit
pour tous entiers $i,n\geq 0$ un homomorphisme de suites exactes scind\'ees: 
\begin{equation} \label{eq:reskth}
\xymatrix{
0 \ar[r] & 
H^{i+1}_{p^n,L,A}(k) \ar[r] \ar[d]^{u_\ast} & 
H^{i+1}_{p^n,L,A}(K) \ar[r]^{\partial^i} \ar[d]^{v_\ast} &
H^i_{p^n,L,A} (k) \ar[r] \ar[d]^{u_\ast} &
0 \\
0 \ar[r] & 
H^{i+1}_{p^n,L,A}(\kappa(v)) \ar[r] & 
H^{i+1}_{p^n,L,A}(F) \ar[r]^{\partial^i} &
H^{i}_{p^n,L,A} (\kappa(v)) \ar[r] &
0.
}
\end{equation}
De plus, $v_\ast$ ne d\'epend pas du choix de $v$.  Si $k=k'$, on a un isomorphisme canonique $H^{i+1}_{n,B_K}(K)\cong H^{i+1}_{n,B_K}(F)$.
\end{corr}

\subsection{Le rel\`evement}

Avant de relever, on d\'emontre un r\'esultat analogue \`a celui de Merkurjev (Proposition \ref{prop:merkurjev}).  Ce sera un corollaire imm\'ediat du Th\'eor\`eme de Kahn suivant qui utilise des groupes de cohomologie de Zariski avec leur variante
r\'eduite:  
\[ \tilde{H}^0_{Zar}(\GB,H^3_{p^n}) \cong H^0_{Zar}(\GB,H^3_{p^n})/H^3_{p^n}(k). \]

\begin{thm}[Kahn \cite{kahnweight2mot}] \label{thm:kahn}
Soient $k$ un corps de $\car(k)=p$,  $\GB$ un $k$-groupe alg\'ebrique simplement connexe, absolument $k$-simple et $\GBbar:=\GB\times_k k_s$ et $n>0$ entier. 
Si $\CH^2(G)=0$, on a une injection
\[
 \tilde{H}^0_{\text{Zar}}(\GB,H^3_{p^n}) \hookrightarrow H^0_{\text{Zar}}(\GBbar,H^3_{p^n}).
\]
\end{thm}

\begin{rem}
 Dans l'enonc\'e (et dans la suite), $G$ (resp. $\Gbar$) signifie $\GB(k)$ (resp. $\GBbar(k_s)$).
\end{rem}

\begin{proof}
Soit $\Gamma=\Gamma_k$ le groupe de Galois absolu de $k$.
Utilisant la cohomologie motivique \`a la Lichtenbaum, Kahn construit 
un morphisme (ibid., premi\`ere suite exacte p.~406)
\begin{equation} \label{eq:morphkahn}
 \tilde{H}^0_{\text{Zar}}(\GB,H^3_{p^n}) \to \Hb^5(\GBbar/k_s,\Gamma(2))^\Gamma
\end{equation}
dont le noyau est contenu dans $H^1(F,H^1_{\text{Zar}}(\GBbar,\mathcal{K}_2))$.  
Ici, $\Hb^5(\GB/k_s,\Gamma(2))$ est un groupe d'hypercohomologie d\'efini par Kahn comme
le (5-i\`eme) groupe de hypercohomologie \'etale d'un complexe relatif bas\'e sur le
complexe $\Gamma(2)$ de Lichtenbaum \cite{lichtgamma2}, et $\mathcal{K}_2$ est le faisceau de Zariski associ\'e au pr\'efaisceau $U\mapsto K_2^Q(U)$ (avec $K_2^Q$ la $K$-th\'eorie de Quillen). 
Afin de d\'efinir le morphisme, il faut que $H^0_{\text{Zar}}(\GBbar,\mathcal{K}_2)\cong K_2^Q(k_s)$: fait qui est d\^u \`a Esnault-Kahn-Levine-Viehweg \cite[Prop. 3.20]{eskalevi}.
    Puisque $H^1_{\text{Zar}}(\GBbar,\mathcal{K}_2)=\Zb$ \cite[Prop. 1']{invcohrostpos}, 
le morphisme \eqref{eq:morphkahn} est injectif.  (Voir les r\'ef\'erences pour plus de d\'etails sur ces objets qu'on utilise ici jusque commes objets auxiliaires.)
Puisque $\CH^2(\Gbar)^\Gamma=0$ \cite[Prop. 3.20]{eskalevi}, l'enonc\'e suit alors de l'injection suivante de Kahn (ibid., suite exacte (18) p.~404): 
\[ \Hb^5(\GBbar/k_s,\Gamma(2))^\Gamma \hookrightarrow H^0_{\text{Zar}}(\GBbar,H^3_{p^n}). \]
\end{proof}

\begin{corr} \label{corr:invnulsau}
Soient $k$ un corps de caract\'eristique $p>0$, $L$ une extension finie s\'eparable de $k$ et $\GB$ un $k$-groupe alg\'ebrique lisse semi-simple simplement connexe absolument $k$-simple tel
que $\CH^2(G)=0$. Alors,  $\Inv^3(\GB,\Hcal^3_{p^n,L})=0$ pour tout entier $n>0$.
\end{corr}

\begin{rem}
Ici, $\Hcal^3_{p^n,L}$ est le module de cycles de la D\'efinition \ref{def:modcycsauv} restreint
\`a $\kcorps$ suivant \S \ref{sec:cohomdef} \ref{sec:coexistence}.  On utilise ici $L$ au lieu de 
$\Lbar$ qui appara\^it dans la D\'efinition \ref{def:modcycsauv} pour ne pas allourdir la notation.
\end{rem}

\begin{proof}
Il suffit de d\'emontrer que  $\tilde{A}^0(\GB,\Hcal^3_{p^n,L})=0$.  Puisque Rost d\'emontre que $A^i(\GB,M_j)\cong H^i_{\text{Zar}}(\GB,M_j)$ pour un module de cycles $M$ et des entiers $i,j$ \cite[Cor. 6.16]{Rostmodcyc}, il suffit de d\'emontrer
que $\tilde{H}^0_\text{Zar}(\GB,\Hcal^3_{p^n,L})=0$.

Soit donc $x\in \tilde{H}^0_\text{Zar}(\GB,\Hcal^3_{p^n,L})\subset \tilde{H}^0_\text{Zar}(\GB,\Hcal^3_{p^n})$.  On sait que $H^3_{p^n}(k(\GB))\to H^3_{p^n}(k_s(\GB))$ se factorise \`a travers $H^3_{p^n}(k(\GB) \otimes L)$.
Alors,  $x\in \ker\left[H^3_{p^n}(k(\GB))\to H^3_{p^n}(k_s(\GB))\right]$, puisque  $x\in H^3_{p^n,L}(k(\GB))$, et donc $x\in \ker\bigl[\tilde{H}^0_\text{Zar}(\GB,\Hcal^3_{p^n})\to \tilde{H}^0_\text{Zar}(\GBbar,\Hcal^3_{p^n})\bigr]$.  Le Th\'eor\`eme \ref{thm:kahn} donne que $x=0$.
\end{proof}

Les arguments utilis\'es dans la d\'emonstration du Th\'eor\`eme \ref{thm:modere} sont de type homologique
et sont bien transf\'erables au cas sauvage rempla\c{c}ant la Proposition \ref{prop:merkurjev} par le 
Corollaire \ref{corr:invnulsau}.  On obtient donc le th\'eor\`eme suivant.

\begin{thm} \label{thm:sauvage}
Soient $k$ un corps de $\car(k)=p>0$, $A$ une $k$-alg\`ebre simple centrale de $\ind_k(A)=p^n$ et $\Lbar$ une extension
finie, galoisienne de $k$ qui d\'eploie $A$.  
Soit $(K,R,k)$ un triplet de Cohen associ\'e \`a $k$ et $(L,S,\Lbar)$ un triplet de Cohen associ\'e \`a $\Lbar$.  Soient  
$B$ la $R$-alg\`ebre d'Azumaya relev\'ee et $\rho'\in \Inv^4(\SKb_1(B_K),\Hcal_{p^n,L,\Bcal_K}^\ast)$.
Alors, il existe un unique $\rho \in \Inv^4(\SKb_1(A),\Hcal_{p^n,L,\Acal}^\ast)$, que l'on appelle \textit{l'invariant sp\'ecialis\'e de $\rho'$},
tel que pour toute extension de Cohen $(K',R',k')$ de $(K,R,k)$
le diagramme suivant commute:
\[ 
\xymatrix{ 
\SKb_1(A)(k') \ar[r]^{\rho_{k'}} & H^4_{p^n,L,A}(k') \ar[d] \\
\ar[u]_{\cong} \SKb_1(B_K)(K') \ar[r]_{\rho'_{K'}} 
  & H^4_{p^n,L,B_K} (K').
}
\]
 
\end{thm}

\begin{rem}
 De nouveau, les modules de cycles $\Hcal_{p^n,L,\Bcal_K}^\ast:=(\Hcal^{j}_{p^n,L,\Bcal_K})_{j\geq 2}$ de base $K$ et $\Hcal_{p^n,L,\Acal}^\ast:=(\Hcal^{j}_{p^n,L,\Acal})_{j\geq 2}$ de base $k$ sont les modules de cycles \'evidents.  Il sont les modules de cycles restreints de la module de cycles $\Hcal_{p^n,L,\Bcal_K}^\ast$
de base $R$ respectivement \`a $K$ et \`a $k$ (selon \S \ref{sec:def} \ref{sec:coexistence}).
\end{rem}

\begin{proof}
Afin de g\'en\'eraliser la d\'emonstration du Th\'eor\`eme \ref{thm:modere}, il faut g\'en\'eraliser le 
Corollaire \ref{corr:invnulmod}.  
Il suffit donc de d\'efinir un diagramme comme \eqref{diag:grand} parce que les autres arguments sont une
simple chasse au diagramme transf\'erable au cas sauvage.  Soit donc $\GB=\SLb_1(A)$. 
On consid\`ere le 
diagramme suivant dont les colonnes sont exactes, et dont la ligne centrale est exacte (Corollaire \ref{corr:invnulsau}):
\begin{equation*} 
\xymatrix{
k^\times \ar[r] \ar[d]^{\cdot [A]}
& k(\GB)^\times  \ar[r]^-{\partial^1} \ar[d]^{\cdot [A_{k(\GB)}]}
& \bigoplus_{x\in \GB^{(1)}} \mathbb{Z}  \ar[d]^{\oplus_{x\in \GB^{(1)}} \cdot [A_{k(x)}]}
\\
 H^3_{p^n}(k) \ar[r] \ar[d]
& H^3_{p^n}(k(\GB)) \ar[r]^-{\partial^{3}} \ar[d]
& \bigoplus_{x\in \GB^{(1)}} H^2_{p^n}(k(x)) \ar[d]
\\ H^3_{p^n,A}(k) \ar[r] 
& H^3_{p^n,A}(k(\GB)) \ar[r]^-{\partial^{3}_{A}}
& \bigoplus_{x\in \GB^{(1)}} H^2_{p^n,A}(k(x)) ,
}
\end{equation*}
De nouveau, l'application $\partial^1$ est le morphisme diviseur, et parce que $\Pic(\GB)=0$ \cite[Lem. 6.9]{sansuc}, $\partial^1$
est surjective.
Il faut aussi noter que $\CH^2(G)=0$ puisque $\GB$ est d\'esormais une forme int\'erieure de $\SLb_m(k)$ avec $m=\deg_k(A)$  \cite{panin}.  La m\^eme chasse au diagramme et la m\^eme construction que dans le cas mod\'er\'e finissent la d\'emonstration.
\end{proof}

On peut donc en d\'eduire que l'invariant de Suslin se g\'en\'eralise.

\begin{corr}
Sous les m\^emes conditions que dans le Th\'eor\`eme \ref{thm:sauvage} l'invariant de
Suslin $\rho_{\text{Sus},B_K}$ induit un unique invariant dans  $\Inv^4(\SKb_1(A),\Hcal_{p^n,L,\Acal}^\ast)$ (satisfaisant 
la propri\'et\'e de rel\`evement),
que l'on appelle l'invariant de Suslin $\rho_{\text{Sus},A}$ de $A$.
\end{corr}

\begin{proof}
Il faut d\'emontrer que $\rho=\rho_{\text{Sus},B_K}$ a ses valeurs dans $\Hcal_{p^n,L,B_K}^4$.
Ceci suit du diagramme commutatif
\[
\xymatrix{ 
\SK_1(B_K) \ar[r]^{\rho_{K}} \ar[d] &  H^4_{p^n,B_K}(K) \ar[d] \\
\SK_1(B_L) \ar[r]^{\rho_{L}} & H^4_{p^n,B_K}(L)
}
\]
et de la trivialit\'e de $\SK_1(B_L)$ ($L$ d\'eploie $B_K$).
\end{proof}

\subsection{Le cas g\'en\'eral} \label{sec:relevgen}

Soient $k$ un corps de caract\'eristique $p>0$ et $A$ une $k$-alg\`ebre simple centrale d'indice arbitraire.
On sait qu'il existe une
$k$-alg\`ebre \`a division $D$, unique \`a isomorphisme pr\`es, qui est Brauer-\'equivalente \`a $A$.  Le th\'eor\`eme de
d\'ecomposition 
de Brauer \cite[Prop. 4.5.16]{gilleszam} donne des $k$-alg\`ebres \`a division $D_{p_i}$ ($1\leq i \leq s$) de $\ind_k(D_{p_i})=p_i^{n_i}$ ($p_i$ des premiers diff\'erents) tel que 
$D \cong D_{p_1} \otimes \ldots \otimes D_{p_s}$.  Cet isomorphisme
implique un isomorphisme (ibid, Ch. 4, Ex. 9)
\begin{equation} \label{eq:sk1deco}
\SKb_1(A) \cong \SKb_1(D) \cong \SKb_1(D_{p_1}) \oplus \ldots \oplus \SKb_1(D_{p_s}).
\end{equation}
Mettons $D_{\not p} = \bigotimes_{p_i\neq p} D_{p_i}$.  On a donc
a fortiori un isomorphisme $\SKb_1(A) \cong \SKb_1(D_p) \oplus \SKb_1(D_{\not p})$.  Posons aussi $A_p:=D_p$ et $A_{\not p}:=D_{\not p}$.  

Afin de d\'efinir l'invariant de Suslin en toute g\'eneralit\'e, nous allons coller l'invariant mod\'er\'e (Th\'eor\`eme 
\ref{thm:modere}) et
sauvage (Th\'eor\`eme \ref{thm:sauvage}) avec cet isomorphisme de $\SKb_1(A)$. 
Il faut donc recoller aussi les deux modules de cycles afin d'obtenir le suivant.

\begin{defin} \label{def:gen}
Soient $(K,R,k)$ un $p$-triplet de Cohen,
 $A$ une $K$-alg\`ebre simple centrale de $\ind_K(A)=r=p^nm$ ($p\nmid m$)
et $\overline{A}$ la $k$-alg\`ebre simple centrale r\'esiduelle.  Soit $\Lbar$ une extension finie
galoisienne de $k$ tel qu'elle est un corps de d\'ecomposition de $\overline{A}_p$ et $(L,S,\Lbar)$ un
$p$-triplet de Cohen associ\'e. 
On d\'efinit le module de 
cycles suivant de base $R$: 
\[ \Hcal_{r,L,\mathcal{A}}^\ast:= \Hcal_{m,\mathcal{A}_{\not p}}^\ast \oplus \Hcal_{p^n,L,\mathcal{A}_p}^\ast.
\]
\end{defin}

Utilisant les Th\'eor\`emes \ref{thm:modere} et \ref{thm:sauvage}, on obtient le th\'eor\`eme suivant.

\begin{thm} \label{cor:gen}
Soient $k$ un corps de $\car(k)=p>0$, $A$ une $k$-alg\`ebre simple centrale de $\ind_k(A)=r$ et $\Lbar$ une extension
finie de $k$ qui d\'eploie $A_p$.  
Soient $(K,R,k)$ un triplet de Cohen associ\'e \`a $k$ et $(L,S,\Lbar)$ un triplet de Cohen associ\'ee \`a $\Lbar$.  Soient  
$B$ la $R$-alg\`ebre d'Azumaya relev\'ee et $\rho'\in \Inv^4\left(\SKb_1(B_K),\Hcal_{r,L,\Bcal_K}^\ast\right)$.
Alors, il existe un unique $\rho \in \Inv^4\left(\SKb_1(A),\Hcal_{r,L,\Acal}^\ast\right)$, que l'on appelle \textit{l'invariant sp\'ecialis\'e
de $\rho'$},
tel que pour toute extension de Cohen $(K',R',k')$ de $(K,R,k)$
le diagramme suivant commute:
\begin{equation} \label{diag:gen}
\xymatrix{ 
\SKb_1(A)(k')  \ar[r]^{\rho_{k'}} & H^4_{r,L,A}(k') \ar[d] \\
\ar[u]_{\cong} \SKb_1(B_K)(K') \ar[r]_{\rho'_{K'}} 
 & H^4_{r,L,B_K} (K').
}
\end{equation}
\end{thm}

En g\'en\'eral on sait donc d\'efinir un invariant de Suslin de $\SKb_1(A)$.

\begin{corr} \label{cor:susgen}
Sous les m\^emes conditions que dans le Th\'eor\`eme \ref{cor:gen}, l'invariant de
Suslin $\rho_{\text{Sus},B_K}$ induit un unique invariant dans $\Inv^4(\SKb_1(A),\Hcal^\ast_{r,L,A})$ (satisfaisant
la propri\'et\'e de rel\`evement),
ce que l'on appelle l'invariant de Suslin $\rho_{\text{Sus},A}$ de $A$.
\end{corr}

\section{Remarques finales} \label{sec:rem}

Finissons ce texte par donner quelques remarques.
Notons tout d'abord que Knus-Merkurjev-Rost-Tignol ont \'egalement d\'efini un invariant de $\SKb_1(A)$ pour
une alg\`ebre de biquaternions $A$ en caract\'eristique quelconque \cite[\SS 17]{kmrt}.  L'auteur rapporte que les deux invariants
sont li\'es.  Il \'etudiera ce sujet dans un prochain papier \cite{wsusbiquat} utilisant le fait qu'ils sont les m\^emes dans le cas
mod\'er\'e \cite[Notes \S 17]{kmrt}, \cite[Thm. 4]{suslin}.
Commentons maintenant sur un autre point de vue sur la
construction et puis sur quelques r\'eflexions de la conjecture de Suslin.

\subsection{Autre point de vue} \label{sec:pointvue}

Il y a une autre fa\c{c}on de regarder \`a la construction utilisant les groupes $A^i, \tilde{A}^0$ et $A^0_{\text{mult}}$
de \S \ref{sec:def} \ref{sec:gersten} et \S \ref{sec:lieninv} \ref{sec:lienmerk}.
Soient $(K,R,k)$ un $p$-triplet de Cohen, $A$ une $k$-alg\`ebre simple centrale de $\ind_k(A)=p^n$, $B$ la $R$-alg\`ebre d'Azumaya
relev\'ee,  $(L,S,\Lbar)$ une extension de Cohen finie galoisienne de $(K,R,k)$ tel que $\Lbar$ deploie $A$ et
$\Hcal^\ast:=\Hcal_{n,L,B_K}^\ast$ le module de cycles avec base
$R$ de la D\'efinition \ref{def:modcycsauvrel}.
Notons $\GBcal:=\SLb_1(B)$ qui est d\'efini,
de fa\c{c}on analogue que $\SLb_1(B_K)$, avec une norme r\'eduite sur $B$ induit par un scindage $B\otimes_R S \cong M_m(S)$ -- voir \cite[Ch. III, \SS 1]{Knus} pour
plus de d\'etails.  La fibre g\'en\'erique $\GBcal_K=\SLb_1(B_K)$ est un ouvert de $\GBcal$. Le ferm\'e compl\'ementaire $Z$ 
est l'image de la fibre sp\'eciale $\GB=\SLb_1(A)$ dans $\GBcal$ sous l'immersion des sch\'emas $\psi:\GB \to \GBcal$. 
Pour tout entier $i\geq 0$, les points de $Z$ de codimension $i+1$ correspondent sous $\psi$  aux points de codimension $i$ dans $\GB$.  De m\^eme fa\c{c}on, $\spec(K)$ est un ouvert de $\spec(R)$ avec ferm\'e compl\'ementaire l'image de $\spec(k)$.
La suite de localisation de Rost \cite[\S 5]{Rostmodcyc} donne donc des
suites exactes:
\begin{equation} \label{diag:a0s}
\xymatrix{
0 \ar[r] & A^0(R, \Hcal^4) \ar[r] \ar[d] &  A^0(K, \Hcal^4) \ar[r] \ar[d] & A^0(k, \Hcal^3) \ar[d] \ar[r] & 0 \\
0 \ar[r] &
 A^0(\GBcal,\Hcal^4) \ar[r] &
A^0(\GBcal_K,\Hcal^4) \ar[r]&
A^0(\GB,\Hcal^3) \ar[r] & \ldots
}
\end{equation}
Le Corollaire \ref{corr:invnulmod} et le Corollaire \ref{corr:invnulsau} (\'etant g\'en\'eralis\'e \`a $\Hcal^\ast$ dans la d\'emonstration du Th\'eor\`eme \ref{corr:invnulsau}) induiquent que $\tilde{A}^0(\GB,\Hcal^3)$ est trivial. 
\`A base du diagramme \eqref{diag:a0s}, le lemme du serpent donne donc un isomorphisme
\[\tilde{A}^0(\GBcal_K,\Hcal^4) \cong
\tilde{A}^0(\GBcal,\Hcal^4)\]
qui respecte les \'el\'ements multiplicatifs.  C'est-\`a-dire d\^u \`a l'isomorphisme de Merkurjev \eqref{eq:invarmerk}, on a \'egalement un isomorphisme
\[ \Inv^4(\GBcal_K,\Hcal^\ast) \cong \tilde{A}^0(\GBcal,\Hcal^4)_\text{mult}. \]
Le groupe \`a droite est d\'efini de fa\c{c}on pareille que pour des
groupes alg\'ebriques dans \S \ref{sec:lieninv} \ref{sec:lienmerk}.
Puisque $\Hcal^\ast$ a base $R$, le morphisme  $\GB \to \GBcal$ de sch\'emas donne \'egalement un 
morphisme
\[ A^0(\GBcal,\Hcal^4) \to A^0(\GB,\Hcal^4)\]
qui donne de la m\^eme fa\c{c}on un morphisme d'invariants
\begin{equation} \label{eq:incques}
 \tilde{A}^0(\GBcal,\Hcal^4)_\text{mult} \to \Inv^4(\GB,\Hcal^\ast).
\end{equation}
Au total, on obtient un diagramme
\[ 
\xymatrix{
\Inv^4(\SKb_1(B_K),\Hcal^\ast) \quad \ar@{^(->}[r]^{\quad \ \pi} \ar@{.>}[d] &  \quad \Inv^4(\GBcal_K,\Hcal^\ast) \ar[d]^\varphi \\
\Inv^4(\SKb_1(A),\Hcal^\ast) \quad \ar@{^(->}[r] & \quad  \Inv^4(\GB,\Hcal^\ast)
}
\]
qui induit l'existence de la fl\^eche point\'ee.  En effet, soient $\rho\in \Inv^4(\SKb_1(B_K),\Hcal^\ast)$ et $(F,S,\Fbar)$ une extension de Cohen de $(K,R,k)$, alors
$(\varphi\circ \pi (\rho))_{\Fbar}$ envoie les commutateurs de $A^\times_{\Fbar}$ \`a 0 puisqu'il correspondent aux commutateurs de $B_F^\times$ 
gr\^ace \`a l'isomorphisme $\SKb_1(A)(\Fbar)\cong \SKb_1(B_K)(F)$ (Corollaire \ref{corr:sk1iso}).  

Dans les Th\'eor\`emes
\ref{thm:sauvage} et \ref{cor:gen}, nous avons aussi construit ce morphisme point\'e de groupes d'invariants mais d'une fa\c{c}on
plus explicite.

\subsection{Conjecture de Suslin} \label{sec:suslin}

Nous finissons ce texte donc par quelques remarques autour de la conjecture de Suslin.
Commen\c{c}ons par rappeler la conjecture.

\begin{conj}[Suslin \cite{suslinconj}]
Soient $k$ un corps et $A$ une $k$-alg\`ebre simple centrale.  Alors,
$\SKb_1(A)=0$ si est seulement si $\ind_k(A)$ est sans facteurs carr\'es.
\end{conj}

Comme mentionn\'ee dans l'introduction, c'est une question de n\'ecessit\'e, puisque Wang
a d\'emontr\'e que pour une alg\`ebre simple centrale $A$, le groupe de Whitehead r\'eduit $\SKb_1(A)$ est
trivial si $\ind_k(A)$ est sans facteurs carr\'es \cite{wang}.  D\^u \`a la d\'ecomposition \eqref{eq:sk1deco}, il suffit de traiter le cas o\`u 
l'indice de $A$ est $p$-primaire pour un premier $p$.  De plus, par la formule de r\'eduction d'indice, il suffit de d\'emontrer que $\SKb_1(A)\neq 0$ si $\ind_k(A)=p^2$ ($p$ premier)\cite[Prop. 4]{blanchet}.

Nous pouvons ajouter une question sur l'invariant de Suslin.

\begin{ques} \label{qu:invtriv}
Soient $k$ un corps et $A$ une $k$-alg\`ebre simple centrale de $\ind_k(A)$ contenant un facteur carr\'e.  Est-ce
que l'invariant de Suslin $\rho_{\text{Sus},A}$ est non trivial?
\end{ques}

\begin{rem}
Bien \'evidemment, une r\'eponse affirmative \`a cette question
impliquerait la conjecture de Suslin.  Alors, on pourrait appeler cette question la version \textit{forte} de la conjecture de Suslin.  
\end{rem}

De nouveau, par la formule de r\'eduction de l'indice, il suffit de r\'epondre \`a la question pour des $k$-alg\`ebres simples centrales
$A$ de $\ind_k(A)=p^2$ ($p$ premier).

 Merkurjev a d\'emontr\'e que la conjecture de Suslin vaut pour
des alg\`ebres simples centrales avec indice divisible par 4 (par exemple
une $k$-alg\`ebre de biquaternions) \cite{mersuslinbiquat}.  Puisque l'invariant
de Suslin pour des biquaternions est injectif dans le cas mod\'er\'e \cite[Thm. 3]{suslin}, il est par construction aussi injectif dans le cas sauvage.  Alors, on obtient que pour
une $k$-alg\`ebre simple centrale $A$ de $\ind_k(A)$ divisible par 4, $\rho_{\text{Sus,A}}$ n'est pas trivial, ind\'ependamment
de $\car(k)$ (en particulier pour les caract\'eristiques sauvages, comme 2).

R\'ecemment Rehman-Tikhonov-Yanchevski\u{\i} ont en plus d\'emontr\'e qu'il suffit de v\'erifier la conjecture de Suslin
pour des alg\`ebres \`a division cycliques.  Il suffit m\^eme de d\'emontrer la conjecture pour une classe d'alg\`ebres \`a division cycliques \'el\'ementaires (des produits tensoriels de deux alg\`ebres cycliques de Dickson) \cite[Thm 0.19 - 0.20]{rehmanea}.

En utilisant des rel\`evements d'alg\`ebres simples centrales de la caract\'eristique 
positive \`a la caract\'eristique 0 comme dans \S \ref{sec:relevbase} \ref{sec:relevasc},
on obtient l'\'enonc\'e de comparaison suivant.  

\begin{prop} \label{prop:susred}
Soient $(K,R,k)$ un $p$-triplet de Cohen, $A$ une $k$-alg\`ebre simple centrale et $B$ la $R$-alg\`ebre
d'Azumaya relev\'ee.  Si la conjecture de Suslin (forte) vaut pour $A$, alors elle vaut aussi pour $B_K$.
\end{prop}

\begin{proof}
Rappelons que $\ind_k(A)=\ind_K(B_K)$.
L'enonc\'e sur la conjecture de Suslin m\^eme suit donc imm\'ediatement du Corollaire \ref{corr:sk1iso}.  L'enonc\'e
sur la conjecture de Suslin forte est vrai, parce que l'invariant de Suslin satisfait/est d\'efini par un morphisme
\[ \Inv^4\left(\SKb_1(B_K),\Hcal_{r,L,\Bcal_K}^\ast\right)\to \Inv^4\left(\SKb_1(A),\Hcal_{r,L,\Acal}^\ast\right).\]
\end{proof}

\begin{rem} Une r\'eciproque de 
la Proposition \ref{prop:susred} est une question ouverte et ne suit pas formellement des d\'efinitions.  En effet, si $\SKb_1(A)=0$, c'est-\`a-dire $\SK_1(A\otimes_k k')=0$ pour
toute extension de corps $k'$ de $k$.  Alors, $\SK_1(B_K \otimes_K K')=\SK_1(A\otimes_k k')=0$ pour toute extension de Cohen $(K',R',k')$ de
$(K,R,k)$.  Mais, il n'est pas s\^ur que \break $\SK_1(B_K\otimes_K F)=0$ pour toute extension $F$ de $K$.  On peut reformuler la situation aussi dans le cadre de \S \ref{sec:pointvue}; la Question \ref{qu:invtriv} se traduit \`a la possible injectivit\'e du morphisme \eqref{eq:incques}.
\end{rem}

Les constructions effectu\'ees par l'auteur ne lui semblent donc pas donner
de fa\c{c}ons imm\'ediates \`a faire des r\'eductions fortes de caract\'eristiques.  Il serait int\'eressant de pouvoir d\'efinir
une des fl\`eches point\'ees (dans un sens au choix) dans le diagramme en bas.  On y abr\`ege la conjecture de Suslin (forte) par CS(F).
\[
\xymatrix{
\fbox{\text{CS en caract\'eristique positive}} \ar@2{<.>}[rr]^{\quad ?} & & \fbox{\text{CS en caract\'eristique 0}} \\
\fbox{\text{CSF en caract\'eristique positive}} \ar@2{<.>}[rr]^{\quad ?} \ar@2{->}[u] & & \fbox{\text{CSF en caract\'eristique 0}} \ar@2{->}[u]
}
\]

\appendix
\section{V\'erification des r\`egles} \label{sec:append}

Dans cette appendice, on v\'erifie que toutes les r\`egles sont bien d\'efinies
pour le module de cycles $\Hcal_{p^n,L}^\ast$ de la Section \ref{sec:sauvage}.  
Nous rappelons les r\`egles pour un module de cycles $M$ de base $R$
et v\'erifions qu'ils vont pour le module de cycle $\Hcal^\ast_{p^n,L}$.
Dans les r\`egles, $E,F,G$ sont des $\rcorps$ arbitraires, et toute application entre des corps est un morphisme de 
$\rcorps$.

\begin{enumerate}[\bf {R1}a:]
 \item \label{it:r1a} Pour tous $\varphi:F\to E, \psi:E\to G$, on a $(\psi \circ \varphi)_\ast=\psi_\ast \circ \varphi_\ast$.
 \item \label{it:r1b} Pour tous $\varphi:F\to E, \psi:E\to G$ finis, on a $(\psi \circ \varphi)^\ast=\varphi^\ast \circ \psi^\ast$.
 \item \label{it:r1c} Soient $\varphi:F\to E, \psi:F\to G$ avec $\varphi$ fini et $S=G\otimes_F E$.  Pour $p\in \spec (S)$, soient
$\varphi_p:G\to S/p, \psi_p:E\to S/p$ les applications naturelles et $l_{p}$ la longueur de l'anneau localis\'e
$S_{(p)}$.  Alors,
\[ \psi_\ast \circ \varphi^\ast = \sum_p l_p \cdot (\varphi_p)^\ast \circ (\varphi_p)_\ast. \]
\end{enumerate}

\begin{enumerate}[\bf R2a:]
 \item[\bf R2:]  \setcounter{enumi}{0} Pour $\varphi:F\to E, x\in K_\ast F, y\in K_\ast E,\rho \in M(F), \mu \in M(E)$, on a 
(avec $\varphi$ fini dans R2\ref{it:r2b} et R2\ref{it:r2c}):
\item $\varphi_\ast (x\cdot \rho)=\varphi_\ast (x) \cdot \varphi_\ast (\rho)$,
\item \label{it:r2b} $\varphi^\ast (\varphi_\ast(x) \cdot \mu) = x \cdot \varphi^\ast(\mu)$, et
\item \label{it:r2c} $\varphi^\ast (y \cdot \varphi_\ast(\rho))= \varphi^\ast (y) \cdot \rho$.
\end{enumerate}

\begin{enumerate}[\bf R3a:]
 \item \label{it:r3a} Soient $\varphi:E\to F$ et $v$ une valuation de $F$ qui se restreint \`a une valuation $w$ non triviale sur $E$
avec indice de ramification $e$.  Soit $\bar{\varphi}:\kappa(w)\to \kappa(v)$ l'application induite.
Alors,
\[ \partial_v \circ \varphi_\ast = e \cdot \bar{\varphi}_\ast \circ \partial_w. \]
\item \label{it:r3b}  Soient $\varphi:F\to E$ fini et $v$ une valuation de $F$.  Pour toute extension $w$ de $v$ sur $E$, soit
$\varphi_w:\kappa(v)\to \kappa(w)$ l'application induite.  Alors,
\[ \partial_v \circ \varphi^\ast = \sum_w \varphi^\ast_w \circ \partial_w. \]
\item \label{it:r3c}  Soient $\varphi:E\to F$  et $v$ une valuation de $F$ qui est triviale sur $E$.  Alors,
\[ \partial_v \circ \varphi_\ast = 0. \] 
\item \label{it:r3d}  Soient $\varphi:E\to F$, $v$ une valuation de $F$ qui est triviale sur $E$, $\bar{\varphi}:E\to \kappa(v)$ l'application
induite et $\pi$ une  uniformisante de $v$.  Soit de plus $s^\pi_v:M(F) \to M(\kappa(v))$ d\'efini par $s^\pi_v(\rho)=\partial_v(\{-\pi \}\cdot \rho)$, alors
\[ s^\pi_v \circ \varphi_\ast = \bar{\varphi}_\ast. \]
\item \label{it:r3e}  Soient $v$ une valuation sur $F$, $u$ une $v$-unit\'e et $\rho\in M(F)$, alors on a
\[ \partial_v(\{-u\}\cdot \rho)=-\{\bar{u} \} \cdot \partial_v(\rho).\]
\end{enumerate}

Pour un $R$-sch\'ema $\Xcal$, on note $M(x)=M(\kappa(x))$ pour $x\in \Xcal$.  Si $\Xcal$ est irr\'eductible, son point g\'en\'erique
est not\'e $\xi$.  Si $\Xcal$ est normal, tout $x\in \Xcal^{(1)}$ induit $\partial_x:M(\xi)\to M(x)$.  Pour tous $x,y\in \Xcal$, on d\'efinit maintenant $\partial^x_y$.  On pose $\partial_y^x=0$ si $Z=\overline{\{x\}}$ et $y\not\in Z^{(1)}$.  Autrement, soit $\tilde{Z}\to Z$ la
normalisation et 
\[ \partial^x_y:=\sum_{z|y} \varphi_z^\ast \circ \partial_{z}, \]
o\`u $z$ parcourt les points de $\tilde{Z}$ au dessus de $y$ et $\varphi_z$ est le morphisme fini $\kappa(y)\to \kappa(z)$.
\begin{enumerate}
 \item[\bf FD:] \textit{(\og Finite support of divisors \fg)} Soient $\Xcal$ un $R$-sch\'ema normal et $\rho\in M(\xi)$.  Alors, $\partial_x(\rho)=0$ pour presque tout $x\in \Xcal^{(1)}$. 
\item[\bf C:]  \textit{(\og Closedness \fg)}  Soient $\Xcal$ int\'egral, local de dimension $2$ et $x_0$ le point ferm\'e de $\Xcal$.  Alors,
\[ 0 = \sum_{x\in \Xcal^{(1)}} \partial_{x_0}^x \circ \partial_x^\xi : M(\xi) \to M(x_0). \] 
\end{enumerate}

\begin{prop}
Soit $(K,R,k)$ un $p$-triplet de Cohen avec $(L,S,\Lbar)$ une extension de Cohen finie galoisienne.  Alors,
$\Hcal^\ast_{p^n,L}$ de la D\'efinition \ref{def:modcycsauv} respecte les r\`egles
R1a-R3e, FD et C de module de cycles ($n>1$) un entier.
\end{prop}

\begin{rem}
Les donn\'ees D1-D4 sont donn\'ees dans \S \ref{sec:sauvagecoh} \ref{sec:diffloggen}, \ref{sec:sauvdef}, \ref{sec:sauvkth} et \ref{sec:sauvres}. 
\end{rem}

\begin{proof}

Les r\`egles R1a-R3e suivent imm\'ediatement de la d\'efinition de $\Hcal^\ast_{p^n,L}$.  \`A noter que la r\`egle R1c suit de la
propri\'ete universelle des produits tensoriels.
La v\'erification de la r\`egle FD suit comme dans le cas classique du support fini des diviseurs \cite[Ch. II. Lem. 6.1]{hart}.

Nous allons d\'eduire la r\`egle $C$  du fait qu'il vaut pour les $K$-groupes de Milnor \cite{katoct}.  Les
r\'esidus $\partial_K$ pour les $K$-groupes de Milnor sont expliqu\'es dans \S \ref{sec:cohomdef} \ref{sec:exmodcyc} et dans \eqref{eq:reskth}.  Pour \'eviter une 
$K$-cophonie, nous supposons pour cette partie que la base de module de cycles est 
$(F,R,\Fbar)$ au lieu de $(K,R,k)$. 
Soit donc $\Xcal$ un $R$-sch\'ema int\'egral et local de dimension $2$.  On suppose tout d'abord que le morphisme
structural $\Xcal$ est surjectif.  Alors, $X:=\Xcal \times_R F$ est un $F$-sch\'ema et $Y:=\Xcal \times_R \Fbar$
est un $\Fbar$-sch\'ema, tous les deux de dimension 1 et $\car(F(X))=0$ et $\car(\Fbar(Y))=p$.  Il faut donc v\'erifier que la composition des r\'esidus
fournit un complexe ($y_0$ le point ferm\'e de $\Xcal$ et $q\geq 2$):
\begin{equation} \label{eq:complexe}
H^{q+1}_{p^n,L}(F(X)) 
\to 
\bigoplus_{x\in X^{(1)}} H^{q}_{p^n,L}(F(x))
\oplus 
\bigoplus_{y\in Y^{(0)}} H^{q}_{p^n,L}(\Fbar(y))
\to 
H^{q-1}_{p^n,L}(\Fbar(y_0)).
\end{equation}

Nous allons d\'ecrire les groupes et les r\'esidus en question avec des $K$-groupes pour pouvoir utiliser la r\`egle $C$
pour les groupes de $K$-th\'eorie.  
 D\'ecrivons d'abord les diff\'erents groupes avec la $K$-th\'eorie
de Milnor.   

\begin{itemize}
 \item \textit{Le groupe $H^{q+1}_{p^n,L}(F(X))$:}\\
Puisque 
\[ \Gamma=\Gal(F_{\text{nr}}(X)/F(X))\cong\Gal(F_{\text{nr}}/F)\cong \Gal(\Fbar_s/\Fbar),\]
on sait que $\cd_p(\Gamma)\leq 1$ \cite[Ch. II, Prop. 3]{serregalcoh}. 
La suite spectrale de Hochschild-Serre
\[ E_2^{s,t} := H^s\left(\Gamma,H^t(F_{\text{nr}}(X),\mu_{p^n}^{\otimes q})\right) \Longrightarrow H^{s+t}(F(X),\mu_{p^n}^{\otimes q}) \] 
induit donc un isomorphisme 
\[ H^1\bigl(\Gamma,H^q(F_{nr}(X),\mu_{p^n}^{\otimes q})\bigr) \cong 
\ker \bigl[ H^{q+1}_{p^n}(F(X)) \to H^{q+1}_{p^n}(F_{\text{nr}}(X))\bigr]. \] 
La conjecture de Bloch-Kato, prouv\'ee par Voevodsky-Rost-Weibel \cite{blochkato,voevodblk,rostblk,weibelblk}, 
dit en plus que $H^q(F_{\text{nr}}(X),\mu_{p^n}^{\otimes q})\cong K_q(F_{\text{nr}}(X))/p^n$.  Ceci nous donne 
un isomorphisme 
\begin{equation} \label{eq:kernr}
H^1\bigl(\Gamma,K_q(F_{\text{nr}}(X))/p^n\bigr) \cong \ker \bigl[ H^{q+1}_{p^n}(F(X)) \to H^{q+1}_{p^n}(F_{\text{nr}}(X))\bigr]
\end{equation}
 et donc une inclusion
\begin{equation} \label{eq:inc1}
H^{q+1}_{p^n,L}(F(X)) \subset H^1(\Gamma,K_q(F_{\text{nr}}(X))/p^n).  
\end{equation}
 
\item \textit{Le groupe $H^{q+1}_{p^n,L}(F(x))$ pour $x\in X^{(1)}$:}\\
De m\^eme fa\c{c}on que ci-dessus, on obtient une inclusion
\begin{equation} \label{eq:inc2}
H^{q}_{p^n,L}(F(x)) \subset H^1(\Gamma,K_{q-1}(F_{\text{nr}}(x))/p^n).  
\end{equation}

\item \textit{Le groupe $H^{q}_{p^n,L}(\Fbar(y))$ pour $y\in Y^{(0)}$:}\\
Soit $y\in Y^{(0)}$, alors $H^{q}_{p^n}(\Fbar(y))\cong H^1 \Bigl(\Fbar(y),\nu_n(q-1)_{\Fbar(y)_s}\Bigr)$.  
L'isomorphisme de Bloch-Kato-Gabber  $\nu_n(q-1)_{\Fbar(y)_s}\cong K_{q-1}(\Fbar(y)_s)/p^n$ \cite[Thm. 2.1]{blochkato} donne donc un isomorphisme,
\[ H^1\bigl(\Fbar(y),K_{q-1}(\Fbar(y)_s)/p^n\bigr) \cong H^{q+1}_{p^n}(\Fbar(y)),\] qui implique aussi une inclusion :
\begin{align} 
  H^{q}_{p^n,L}(\Fbar(y)) & \cong \ker \bigl[ H^1\bigl(\Fbar(y),K_{q-1}(\Fbar(y)_s)/p^n\bigr) \to H^1\bigl(\Lbar(y),K_{q-1}(\Fbar(y)_s)/p^n\bigr) \bigr] \notag \\
& \subset 
\ker \bigl[ H^1\bigl(\Fbar(y),K_{q-1}(\Fbar(y)_s)/p^n\bigr) \to H^1\bigl(\Fbar_s(y),K_{q-1}(\Fbar(y)_s)/p^n\bigr)\bigr]. \label{eq:inc3}
\end{align}
Ce dernier est isomorphe \`a $H^1\bigl(\Gamma,(K_{q-1}(\Fbar(y)_s)/p^n)^{\Gamma_{\Fbar_s(y)}}\bigr)$ par la suite d'inflation-restriction \cite[Prop. 3.3.14]{gilleszam}.  
  
\item \textit{Le groupe $H^{q}_{p^n,L}(\Fbar(y_0))$ pour $y_0$ le point ferm\'e de $\Xcal$:}\\
Alors comme au-dessus:
\begin{equation} \label{eq:inc4}
H^{q-1}_{p^n,L}(\Fbar(y_0)) \subset H^1\left(\Gamma,(K_{q-2}(\Fbar(y_0)_s)/p^n)^{\Gamma_{\Fbar_s(y_0)}}\right). 
\end{equation}\vspace{5mm}
\end{itemize}
Expliquons maintenant les r\'esidus en termes de $K$-th\'eorie.
\begin{itemize}
 \item \textit{Le r\'esidu $\partial_x:H^{q+1}_{p^n,L}(F(X))\to H^{q+1}_{p^n,L}(F(x))$ pour $x\in X^{(1)}$:}\\
La valuation associ\'ee \`a $x$ induit bien le r\'esidu $\partial_x$, mais aussi un  r\'esidu $\Gamma$-\'equivariant $\partial_{K,x}:K_q(F_{\text{nr}}(X))/p^n\to K_{q-1}(F_{\text{nr}}(x))/p^n$ (puisque $ \Gal (F_{\text{nr}}(x)/F(x))\cong \Gamma$). 
Ceci induit donc un morphisme (\`a qui on donne le m\^eme nom par abus de notation):
\[ \partial_{K,x}: H^1(\Gamma,K_q(F_{\text{nr}}(X))/p^n)\to H^1(\Gamma,K_{q-1}(F_{\text{nr}}(x))/p^n).\]
 Le Lemme
\ref{lem:kcom} plus loin induit que $\partial_{K,x}$ est compatible avec $\partial_x$ par les inclusions \eqref{eq:inc1} et \eqref{eq:inc2}, c'est-\`a-dire on a un diagramme commutatif:
\begin{equation} \label{diag:sit1}
\xymatrix{
H^{q+1}_{p^n,L}\left(F(X)\right)\ \ \ar@{^(->}[r] \ar[d]_{\partial_x} & \ \ H^1\bigl(\Gamma,K_q(F_{\text{nr}}(X))/p^n\bigr) \ar[d]^{\partial_{K,x}} \\  
H^{q}_{p^n,L}\left(F(x)\right)\ \ \ar@{^(->}[r] &\ \ H^1\bigl(\Gamma,K_{q-1}(F_{\text{nr}}(x))/p^n\bigr).  
}
\end{equation}

\item \textit{Le r\'esidu $\partial_y:H^{q+1}_{p^n,L}(F(X))\to H^{q}_{p^n,L}(\Fbar(y))$ pour $y\in Y^{(0)}$:}\\
La valuation associ\'ee \`a $y$ induit bien un r\'esidu  $\partial_y$.  Dans le Lemme \ref{lem:kcom},
on d\'emontre que sous l'injection \eqref{eq:inc3} $\im (\partial_y)$ est  envoy\'e dans  $H^1\left(\Gamma,K_q(\Fbar_s(y))/p^n\right)$.  
De l'autre c\^ot\'e, la valuation associ\'ee \`a $y$ induit un r\'esidu $\Gamma$-\'equivariant
$\partial_{K,y}:K_q(F_{\text{nr}}(X))\to K_{q-1}\left(\Fbar_s(y)\right)$ et donc un morphisme:
\[ \partial_{K,y}: H^1\bigl(\Gamma,K_{q}(F_{\text{nr}}(X))/p^n\bigr) \to H^1\bigl(\Gamma,K_{q-1}(\Fbar_s(y))/p^n\bigr).\]
Le Lemme \ref{lem:kcom} d\'emontre qu'on a un diagramme commutif qui exprime la compatibilit\'e de $\partial_y$ et de $\partial_{K,y}$ sous les inclusions \eqref{eq:inc1} et \eqref{eq:inc3}:
\begin{equation} \label{diag:sit2}
\xymatrix{
\ \ H^{q+1}_{p^n,L}\left(F(X)\right) \ \ \ar@{^(->}[r] \ar[d]_{\partial_y} & \ \ H^1\bigl(\Gamma,K_q(F_{\text{nr}}(X))/p^n\bigr) \ \ \ar[d]^{\partial_{K,y}} \\  
\ \ H^{q}_{p^n,L}\left(\Fbar(y)\right) \ \ \ar@{^(->}[r] & \ \ H^1\bigl(\Gamma,K_{q-1}(\Fbar_s(y))/p^n\bigr).  \ \ 
}
\end{equation}

 \item \textit{Le r\'esidu $\partial^x_{y_0}:H^{q+1}_{p^n,L}(F(x))\to H^{q+1}_{p^n,L}(F(y_0))$ pour $x\in X^{(1)}$:}\\
De nouveau, on a bien le r\'esidu $\partial^x_{y_0}$, et dans le Lemme \ref{lem:kcom} on d\'emontre que sous l'inclusion \eqref{eq:inc4} $\im (\partial^x_{y_0})$ 
est contenu dans $H^1\left(\Gamma,K_{q-2}(\Fbar_s(y_0))/p^n\right)$.
Par ailleurs, on a aussi un r\'esidu $\Gamma$-\'equivariant $\partial_{K,y_0}^x:K_{q-1}(F_{\text{nr}}(x))\to K_{q-2}(\Fbar_s(y_0))$ qui donne au
niveau de cohomologie un morphisme :
\[ \partial_{K,y_0}^x:H^1(\Gamma,K_{q-1}(F_{\text{nr}}(x))/p^n) \to H^1(\Gamma,K_{q-2}(\Fbar_s(y_{0}))/p^n). \]
De nouveau, le Lemme \ref{lem:kcom} garantit que $\partial_{K,y_0}^x$ est compatible avec $\partial^x_{y_0}$ sous les
inclusions \eqref{eq:inc2} et \eqref{eq:inc4} formant ainsi le diagramme commutatif:
\begin{equation} \label{diag:sit3}
\xymatrix{ 
\ \  H^{q}_{p^n,L}(F(x)) \ \  \ar[d]_{\partial^x_{y_0}} \ar@{^(->}[r] & \ \  H^1\bigl(\Gamma,K_{q-1}(F_{\text{nr}}(x))/p^n\bigr) \ar[d]^{\partial^x_{K,y_0}} \ \  \\
\ \  H^{q-1}_{p^n,L}(\Fbar(y_0)) \ \  \ar@{^(->}[r] & \ \  H^1\bigl(\Gamma,K_{q-2}(\Fbar_s(y_0))/p^n\bigr).  \ \ 
}
\end{equation}

\item \textit{Le r\'esidu $\partial^y_{y_0}:H^{q}_{p^n,L}(\Fbar(y))\to H^{q+1}_{p^n,L}(F(y_0))$ pour $y\in Y^{(0)}$:}\\
Dans cette situation, on a aussi  un r\'esidu $\partial^y_{y_0}$ sur les groupes de cohomologie et un r\'esidu 
$\Gamma$-\'equivariant de la $K$-th\'eorie $\partial_{K,y_0}^y: K_{q-1}(\Fbar_s(y))\to K_{q-2}(\Fbar_s(y_0))$  (pour $y\in Y^{(1)}$).  Alors, 
$\partial_{K,y_0}^y$ induit un morphisme au niveau de cohomologie :
\[ \partial_{K,y_0}^y:H^1(\Gamma,K_{q-1}(\Fbar_s(y))/p^n) \to H^1(\Gamma,K_{q-2}(\Fbar_s(y_{0}))/p^n),  \]
et le Lemme \ref{lem:kcom} d\'emontre la compatibilit\'e de $\partial_{K,y_0}^y$ avec $\partial_{y_0}^y$ sous les 
inclusions \eqref{eq:inc3} et \eqref{eq:inc4}:
\begin{equation} \label{diag:sit4}
 \xymatrix{
\ \  H^{q}_{p^n,L}(\Fbar(y)) \ \ \ar[d]_{\partial_{y_0}^y} \ar@{^(->}[r] & \ \ H^1 \bigl(\Gamma,K_{q-1}(\Fbar_s(y))/p^n\bigr)\ \  \ar[d]^{\partial_{K,y_0}^y} \\ 
\ \ H^{q-1}_{p^n,L}(\Fbar(y_0)) \ \ \ar@{^(->}[r] & \ \ H^1\bigl(\Gamma,K_{q-2}(\Fbar_s(y_0))/p^n\bigr). \ \  
}
\end{equation}
\end{itemize}
En somme, on a donc un ensemble de r\'esidus,
\begin{multline*} 
 H^1\bigl(\Gamma,K_q(F_{\text{nr}}(X))/p^n\bigr)   \to 
\bigoplus_{x\in X^{(1)}} H^1\bigl(\Gamma,K_{q-1}(F_{\text{nr}}(x))/p^n\bigr) \oplus \bigoplus_{y\in Y^{(0)}} H^1(\Gamma,K_{q-1}\bigl(\Fbar_s(y))/p^n\bigr) \\
\to H^1\bigl(\Gamma,K_{q-2}(\Fbar_s(y_{0}))/p^n\bigr), 
\end{multline*}
dont on sait que c'est un complexe parce que les $K$-groupes de Milnor respectent la r\`egle $C$ \cite{katoct}.  
Les diagrammes commutatifs (\ref{diag:sit1},\ref{diag:sit2},\ref{diag:sit3},\ref{diag:sit4}) 
donnent bien que \eqref{eq:complexe} est un complexe.

Si le morphisme structural n'est pas surjectif, on a ou bien un $F$-sch\'ema, ou bien un $\Fbar$-sch\'ema.  Si $\Xcal$
est un $F$-sch\'ema, les groupes en vigueur sont d\'efinis commes des noyaux des groupes de la cohomologie galoisienne mod\'er\'e.  La r\`egle
$C$ suit donc de la r\`egle $C$ du cas mod\'er\'e.  Si $\Xcal$ est un $\Fbar$-sch\'ema,  
on r\'e\'ecrit \eqref{eq:complexe} avec \eqref{eq:isonu} et l'isomorphisme de Bloch-Gabber-Kato comme
\[ 
H^1\bigl(\Gamma,K_q(\Fbar_s(\Xcal))/p^n\bigr) \to 
\bigoplus_{x\in \Xcal^{(1)}} H^1\bigl(\Gamma,K_{q-1}(\Fbar_s(x))/p^n\bigr) \to
H^1\bigl(\Gamma,K_{q-1}(\Fbar_s(x_0))/p^n\bigr),
\]
avec $x_0$ le point ferm\'e de $\Xcal$.
Celui est de nouveau un complexe, parce que les r\'esidus sont compatibles avec les r\'esidus de la $K$-th\'eorie de
Milnor (voir Lemme \ref{lem:kcom} dans le cas ``$y$ et $y_0$''), et parce que la r\`egle $C$ vaut pour la $K$-th\'eorie de Milnor \cite{katoct}. 
\end{proof}

\begin{lemma} \label{lem:kcom}
Soit $\Xcal$ un $R$-sch\'ema tel que le morphisme structural est surjectif, alors les diagrammes
(\ref{diag:sit1},\ref{diag:sit2},\ref{diag:sit3},\ref{diag:sit4}) sont commutatifs.
\end{lemma}

\begin{proof}
On a quatre situations; traitons-les cas par cas.
\begin{itemize}
 \item \textit{Le diagramme \eqref{diag:sit1} est commutatif pour $x\in X^{(1)}$:}\\
L'isomorphisme $K_q(F_{nr}(X))/p^n \cong H^q(F_{\text{nr}}(X),\mu_{p^n}^{\otimes q})$
de Bloch-Kato est d\'efini par le symbole galoisien.  Cet isomorphisme commute avec le r\'esidu usuel sur\break $H^q(F_{\text{nr}}(X),\mu_{p^n}^{\otimes q})$
(avec une section donn\'ee par le cup-produit par la classe d'une uniformisante $\pi_x$ pour la valuation associ\'ee \`a $x$) \cite[Prop. 7.5.1]{gilleszam}.  On en d\'eduit 
le r\'esultat, puisque l'isomorphisme \eqref{eq:kernr} est l'inflation, 
et puisque $\partial_{x}$ a aussi une section donn\'ee par le cup-produit par la classe de $\pi_x$.

\item \textit{Le diagramme \eqref{diag:sit2} est commutatif pour $y\in Y^{(0)}$:}\\
Il faut donc aussi v\'erifier que $\im (\partial_{y})$ est contenu dans $H^1\left(\Gamma,K_{q-1}(\Fbar_{s}(y))/p^n\right)$.
Parce que le r\'esidu $\partial_{y}$ est d\'efini par une section, on peut prendre 
$w\otimes \bar{x}_2 \otimes \ldots \otimes \bar{x}_q\in H^q_{p^n,L}(\Fbar(y))$ avec $w\in W_n(\Fbar(y))$ et
$x_2,\ldots,x_q\in \Ocal_y^\times$ ($\Ocal_y$ \'etant l'anneau de valuation associ\'e \`a la valuation induite
par $y$).  Si $\pi_y$ est une uniformisante de $F(X)$ pour
la valuation associ\'ee \`a $y$, il est le r\'esidu de 
\[ i(w)\cup h^q_{p^n,F(X)}(\{\pi_y,x_2,\ldots x_q\})\in H^{q+1}_{p^n,L}\left(F(X)\right).\]  
Il correspond donc \`a 
\[\bigl((\sigma(a)-a)\{\pi_y,x_2,\ldots,x_q\}\bigr)_\sigma \in H^1\left(\Gamma,K_q(F_{\text{nr}}(X))/p^n\right),\] 
o\`u $a^{(p)} -a = w$ avec $a\in W_n(\Fbar(y))$ et o\`u on consid\`ere $(\sigma(a)-a)$ comme \'element de $\Zb/p^n\Zb$.  Par ailleurs, $w\otimes \bar{x}_2 \otimes \ldots \otimes \bar{x}_q$ correspond \`a
\[ \bigl((\sigma(a)-a)\{\bar{x}_2,\ldots ,\bar{x}_q\}\bigr)_\sigma  \in H^1(\Gamma,K_q(\Fbar(y)_s)/p^n).\]  La commutativit\'e
suit donc, et il est aussi clair que $\left((\sigma(a)-a)\{\bar{x}_2,\ldots ,\bar{x}_q\}\right)_\sigma$ est en effet
un \'el\'ement de $H^1\left(\Fbar(y),K_q\left(\Fbar_s(y)\right)/p^n\right)$, parce que $\partial_{K,y}$ tombe dedans.

 \item \textit{Le diagramme \eqref{diag:sit3} est commutatif pour $x\in X^{(1)}$:}\\
On sait le v\'erifier de fa\c{c}on analogue au cas pr\'ecedent. 

\item \textit{Le  diagramme \eqref{diag:sit4} est commutatif pour $y\in Y^{(0)}$:}\\
Les isomorphismes \[ \nu_n(q-1)_{\Fbar(y)_s}\cong K_{q-1}(\Fbar(y)_s)/p^n,  \qquad  \nu_n(q-2)_{\Fbar(y_0)_s}\cong K_{q-2}(\Fbar(y_0)_s)/p^n,\] et
le r\'esidu $K_{q-1}(\Fbar(y)_s)\to K_{q-2}(\Fbar(y_0)_s)$ induisent un r\'esidu,
\begin{eqnarray*}
 \nu_n(q-1)_{\Fbar(y)_s} &\to& \nu_n(q-2)_{\Fbar(y_0)_s}, \quad \text{ defini par } \\
 a \otimes \pi_0 \otimes x_2 \otimes \ldots \otimes x_{q-1} & \mapsto  &\bar{a} \otimes \bar{x}_2 \ldots \otimes \bar{x}_{q-1}.
\end{eqnarray*}
 Ici, $a\in W_n(\Ocal_v$) et $x_i\in \Ocal_{v}^\times$, o\`u $\Ocal_{v}$ est l'anneau de valuation associ\'e \`a la valuation $v$ induite
par $y_0$ avec  $\pi_{0}$ une uniformisante.
Par la d\'efinition du r\'esidu $\partial_{y_0}^y$ (voir les Remarques \ref{rem:ressauv} et \ref{rem:ressauvnoncom}), il est clair que les r\'esidus
sont compatibles.
\end{itemize}
\end{proof}

\fontsize{10}{12} \selectfont
\bibliographystyle{alphanum-fr}
\addcontentsline{toc}{section}{R\'ef\'erences} 
\bibliography{bib-ska1}

\begin{thebibliography}{KMRT}

\bibitem[AG]{ausgold}
Maurice Auslander et Oscar Goldman.
\newblock The {B}rauer group of a commutative ring.
\newblock {\em
  \href{http://www.jstor.org/stable/1993378?origin=crossref}{Trans. Amer. Math.
  Soc.}}, 97, 367--409, 1960.

\bibitem[Alb]{albertcycl}
Adrian Albert.
\newblock Simple algebras of degree {$p\sp e$} over a centrum of characteristic
  {$p$}.
\newblock {\em
  \href{http://www.jstor.org/stable/1989665?origin=crossref}{Trans. Amer. Math.
  Soc.}}, 40(1), 112--126, 1936.

\bibitem[BK]{blochkato}
Spencer Bloch et Kazuya Kato.
\newblock {$p$}-adic \'etale cohomology.
\newblock {\em
  \href{http://www.numdam.org/item?id=PMIHES_1986__63__107_0}{Publ. Math. Inst.
  Hautes {\'E}tudes Sci.}}, (63), 107--152, 1986.

\bibitem[Bla]{blanchet}
Altha Blanchet.
\newblock Function fields of generalized {B}rauer-{S}everi varieties.
\newblock {\em
  \href{http://www.informaworld.com/smpp/content~db=all?content=10.1080/009278%
79108824131}{Comm. Algebra}}, 19(1), 97--118, 1991.

\bibitem[BT]{basstate}
Hyman Bass et John Tate.
\newblock The {M}ilnor ring of a global field.
\newblock In {\em Algebraic {$K$}-theory, {II}: ``{C}lassical'' algebraic
  {$K$}-theory and connections with arithmetic ({P}roc. {C}onf., {S}eattle,
  {W}ash., {B}attelle {M}emorial {I}nst., 1972)}, pages 349--446. Lecture Notes
  in Math., Vol. 342. Springer, Berlin, 1973.

\bibitem[Car]{cartier}
Pierre Cartier.
\newblock Questions de rationalit\'e des diviseurs en g\'eom\'etrie
  alg\'ebrique.
\newblock {\em
  \href{http://www.numdam.org/numdam-bin/fitem?id=BSMF_1958__86__177_0}{Bull.
  Soc. Math. France}}, 86, 177--251, 1958.

\bibitem[Coh]{cohen}
Irvin Cohen.
\newblock On the structure and ideal theory of complete local rings.
\newblock {\em
  \href{http://www.jstor.org/stable/1990313?origin=crossref}{Trans. Amer. Math.
  Soc.}}, 59, 54--106, 1946.

\bibitem[EKLV]{eskalevi}
H{\'e}l{\`e}ne Esnault, Bruno Kahn, Marc Levine, et Eckart Viehweg.
\newblock The {A}rason invariant and mod {$2$} algebraic cycles.
\newblock {\em
  \href{http://www.ams.org/jams/1998-11-01/S0894-0347-98-00248-3/home.html}{J.
  Amer. Math. Soc.}}, 11(1), 73--118, 1998.

\bibitem[Gil1]{invcohrostpos}
Philippe Gille.
\newblock Invariants cohomologiques de {R}ost en caract{\'e}ristique positive.
\newblock {\em K-Theory}, 21, 57--100, 2000.

\bibitem[Gil2]{gillebourbakifr}
Philippe Gille.
\newblock Le probl{\`e}me de {K}neser-{T}its.
\newblock {\em S{\'e}minaire {B}ourbaki \no 983, {\`a} para{\^\i}tre \`a
  {A}st{\'e}risque}, 2009.

\bibitem[GMS]{cohinv}
Skip Garibaldi, Alexander Merkurjev, et Jean-Pierre Serre.
\newblock {\em {Cohomological invariants in {Galois} cohomology}}, volume~28 de
  {\em University Lecture Series}.
\newblock Amer. Math. Soc., 2003.

\bibitem[Gro1]{ega4}
Alexander Grothendieck.
\newblock {\em {\'El\'ements de G\'eom\'etrie Alg\'ebrique IV, \'Etude locale
  des sch\'emas et des morphismes de sch\'emas, Premi\`ere Partie}}, volume~20
  de {\em
  \href{http://www.numdam.org/numdam-bin/fitem?id=PMIHES_1964__20__5_0}{Publ.
  Math. Inst. Hautes {\'E}tudes Sci.}}
\newblock Bures-sur-Yvette, 1964.

\bibitem[Gro2]{grothbrauer}
Alexander Grothendieck.
\newblock Le groupe de {B}rauer : {I}. {A}lg\`ebres d'{A}zumaya et
  interpr\'etations diverses.
\newblock {\em
  \href{http://www.numdam.org/numdam-bin/fitem?id=SB_1964-1966__9__199_0}{S\'e%
minaire Bourbaki}}, 9, 199--219, 1964-1966.
\newblock Expos\'e No. 290.

\bibitem[GS]{gilleszam}
Philippe Gille et Tam{\'a}s Szamuely.
\newblock {\em Central Simple Algebras and Galois Cohomology}, volume 101 de
  {\em Cambridge studies in advanced mathematics}.
\newblock Cambridge University Press, Cambridge, 2006.

\bibitem[Har]{hart}
Robin Hartshorne.
\newblock {\em Algebraic Geometry}, volume~52 de {\em Graduate Texts in
  Mathematics}.
\newblock Springer Science+Business Media, Inc., New York, 1977.

\bibitem[Izh]{izboldhin}
Oleg Izhboldin.
\newblock On the cohomology groups of the field of rational functions.
\newblock In {\em Mathematics in {S}t.\ {P}etersburg}, volume 174 de {\em Amer.
  Math. Soc. Transl. Ser. 2}, pages 21--44. Amer. Math. Soc., Providence, RI,
  1996.

\bibitem[Kah]{kahnweight2mot}
Bruno Kahn.
\newblock Applications of weight-two motivic cohomology.
\newblock {\em \href{http://www.math.uiuc.edu/documenta/vol-01/17.html}{Doc
  Math. J. DMV}}, 1, 395--416, 1996.

\bibitem[Kat1]{katogalcoh}
Kazuya Kato.
\newblock Galois cohomology of complete discrete valuation fields.
\newblock In {\em Algebraic {K-T}heory}, volume 967 de {\em Lecture notes in
  mathematics}, pages 215--238, Berlin, 1982.

\bibitem[Kat2]{katoct}
Kazuya Kato.
\newblock A {H}asse principle for two-dimensional global fields.
\newblock {\em
  \href{http://www.digizeitschriften.de/index.php?id=loader&tx_jkDigiTools_pi1%
[IDDOC]=509447}{J. Reine Angew. Math.}}, 366, 142--183, 1986.

\bibitem[KMRT]{kmrt}
Max-Albert Knus, Alexander Merkurjev, Markus Rost, et Jean-Pierre Tignol.
\newblock {\em The book of involutions}, volume~44 de {\em Amer. Math. Soc.
  Colloq. Publ.}
\newblock 1998.

\bibitem[Knu]{Knus}
Max-Albert Knus.
\newblock {\em Quadratic and {H}ermitian forms over rings}, volume 294 de {\em
  Grundlehren der Mathematischen Wissenschaften}.
\newblock Springer-Verlag, Berlin, 1991.

\bibitem[Lic]{lichtgamma2}
Stephen Lichtenbaum.
\newblock The construction of weight-two arithmetic cohomology.
\newblock {\em
  \href{http://www.digizeitschriften.de/index.php?id=loader&tx_jkDigiTools_pi1%
[IDDOC]=377062}{Invent. math.}}, 88, 183--215, 1987.

\bibitem[Mer1]{invalggroup}
Alexander Merkurjev.
\newblock Invariants of algebraic groups.
\newblock {\em
  \href{http://www.reference-global.com/doi/abs/10.1515/crll.1999.508.127}{J.
  reine angew. Math.}}, 508, 127--156, 1999.

\bibitem[Mer2]{mersuslinbiquat}
Alexander Merkurjev.
\newblock The group {$SK\sb 1$} for simple algebras.
\newblock {\em
  \href{http://www.ams.org/leavingmsn?url=http://dx.doi.org/10.1007/s10977-006%
-0021-4}{$K$-Theory}}, 37(3), 311--319, 2006.

\bibitem[Mil]{milnor}
John Milnor.
\newblock Algebraic {$K$}-theory and quadratic forms.
\newblock {\em
  \href{http://www.digizeitschriften.de/index.php?id=loader&tx_jkDigiTools_pi1%
[IDDOC]=374693}{Invent. Math.}}, 9, 318--344, 1969/1970.

\bibitem[NM]{nakmat}
Tadasi Nakayama et Yoz{\^o} Matsushima.
\newblock \"{U}ber die multiplikative {G}ruppe einer {$p$}-adischen
  {D}ivisionsalgebra.
\newblock {\em
  \href{http://projecteuclid.org/DPubS?service=UI&version=1.0&verb=Display&han%
dle=euclid.pja/1195573246}{Proc. Imp. Acad. Tokyo}}, 19, 622--628, 1943.

\bibitem[Pan]{panin}
Ivan Panin.
\newblock Splitting principle and {$K$}-theory of simply connected semisimple
  algebraic groups.
\newblock {\em Algebra i Analiz}, 10(1), 88--131, 1998.

\bibitem[Pla]{sk1niettriv}
Vladimir Platonov.
\newblock The {Tannaka-Artin problem and reduced K-}theory.
\newblock {\em Math. USSR Izv.}, 10(2), 211--243, 1976.
\newblock Traduction anglaise.

\bibitem[Ros1]{Rostmodcyc}
Markus Rost.
\newblock Chow {Groups with C}oefficients.
\newblock {\em \href{http://www.math.uiuc.edu/documenta/vol-01/16.html}{Doc.
  Math. J. DMV}}, 1, 319--393, 1996.

\bibitem[Ros2]{rostblk}
Markus Rost.
\newblock The basic correspondence of a splitting variety.
\newblock \href{http://www.math.uni-bielefeld.de/~rost/basic-corr.html}{Notes
  t\'el\'echargables de son site personnel}, 1998.

\bibitem[RTY]{rehmanea}
Ulf Rehman, Sergey Tikhonov, et Vyacheslav Yanchevski\u{\i}.
\newblock Symbols and cyclicity of algebras after a scalar extension.
\newblock {\em
  \href{http://www.mathematik.uni-bielefeld.de/LAG/man/315.html}{Preprint}},
  2008.

\bibitem[San]{sansuc}
Jean-Jacques Sansuc.
\newblock Groupe de {B}rauer et arithm\'etique des groupes alg\'ebriques
  lin\'eaires.
\newblock {\em
  \href{http://www.digizeitschriften.de/resolveppn/GDZPPN002198746}{J. reine
  angew. Math.}}, 327, 12--80, 1981.

\bibitem[Sch]{schoeller}
Colette Schoeller.
\newblock Groupes affines, commutatifs, unipotents sur un corps parfait.
\newblock {\em
  \href{http://www.numdam.org/numdam-bin/fitem?id=BSMF_1972__100__241_0}{Bulle%
tin de la S.M.F.}}, 100, 241--300, 1972.

\bibitem[Ser1]{serrecorloc}
{J}ean-{P}ierre Serre.
\newblock {\em Corps Locaux}.
\newblock Publications de l'{I}nstitut de {M}ath\'ematique de l'{U}niversit\'e
  de Nancago. Hermann, Paris, 1968.

\bibitem[Ser2]{serregalcoh}
Jean-Pierre Serre.
\newblock {\em Galois {Cohomology}}.
\newblock Springer Monographs in Mathematics. Springer-Verlag, Berlin, 2002.

\bibitem[Sus1]{suslinconj}
Andrei Suslin.
\newblock {$SK\sb 1$} of division algebras and {G}alois cohomology.
\newblock In {\em Algebraic {$K$}-theory}, volume~4 de {\em Adv. Soviet Math.},
  pages 75--99. Amer. Math. Soc., Providence, RI, 1991.

\bibitem[Sus2]{suslin}
Andrei Suslin.
\newblock {\em {$\SK_1(A)$ of division algebras and {Galois} cohomology
  revisited}}, volume 219 de {\em Math. Soc. Transl. Ser. 2}, pages 125--147.
\newblock Am. Math. Soc., 2006.

\bibitem[Voe]{voevodblk}
Vladimir Voevodsky.
\newblock On {Motivic C}ohomology with $\mathbb{Z}/l$ coefficients.
\newblock {\em
  \href{http://www.math.uiuc.edu/K-theory/0639/}{Pr\'epublication}}, 2003.

\bibitem[Wad]{wadsworth}
Adrian Wadsworth.
\newblock Valuation theory on finite dimensional division algebras.
\newblock In {\em Valuation theory and its applications, {V}ol. {I}
  ({S}askatoon, {SK}, 1999)}, volume~32 de {\em Fields Inst. Commun.}, pages
  385--449. Amer. Math. Soc., Providence, RI, 2002.

\bibitem[Wan]{wang}
Shianghaw Wang.
\newblock On the commutator group of a simple algebra.
\newblock {\em \href{http://www.jstor.org/stable/2372036?origin=crossref}{Amer.
  J. Math.}}, 72, 323--334, 1950.

\bibitem[Wei]{weibelblk}
Charles Weibel.
\newblock {\em 2007 Trieste Lectures on The Proof of the Bloch-Kato
  Conjecture}, volume~23 de {\em ICTP Lecture Notes Series}, pages 1--28.
\newblock 2008.

\bibitem[Wit]{witt}
Ernst Witt.
\newblock Zyklische {K\"orper und Algebren der Charakteristic $p$ vom Grad}
  $p^n$.
\newblock {\em
  \href{http://www.digizeitschriften.de/resolveppn/GDZPPN002173824}{J. reine
  angew. Math.}}, 176, 126--140, 1937.

\bibitem[Wou]{wsusbiquat}
Tim Wouters.
\newblock On {S}uslin's invariant for biquaternion algebras.
\newblock {\em En pr\'eparation}, 2009.

\end{thebibliography}

\end{document}